\documentclass[12pt]{amsart}
\usepackage[utf8]{inputenc}
\usepackage{fullpage}
\usepackage{fontenc}
\usepackage{amsfonts}
\usepackage{amssymb}
\usepackage{amsmath}
\usepackage{amsthm}\usepackage{mathtools}
\usepackage{enumerate}
\usepackage{hyperref}\usepackage{enumitem}
\usepackage{mathrsfs}
\usepackage{tikz}\usetikzlibrary{calc}
\usepackage{marginnote}
\usepackage{xcolor,enumitem}
\usepackage{soul}\usepackage{tikz}
\usetikzlibrary{positioning,calc}
\usepackage[all,cmtip]{xy}

\newcommand{\R}{\mathbb{R}}

\newcommand{\N}{\mathbb{N}}
\newcommand{\C}{\mathbb{C}}

\newcommand{\cV}{\mathcal{V}}

\newcommand{\cA}{\mathcal{A}}
\newcommand{\cR}{\mathcal{R}}

\newcommand{\cE}{\mathcal{E}}
\newcommand{\cQ}{\mathcal{Q}}
\newcommand{\eps}{\varepsilon}
\newcommand{\REMARK}[1]{\marginpar{\tt #1}}
\newtheorem*{acknowledgements}{Acknowledgements}

\newcommand{\cstu}{\mathrm{C}^*_u}

\newtheorem*{rigprob*}{Rigidity Problem for uniform Roe Algebras}
\newtheorem*{rigprobcorona*}{Rigidity Problem for uniform Roe Coronas}

\newcommand{\Cstar}{\mathrm{C}^*}
\newcommand{\cst}{\mathrm{C}^*}
\newcommand{\cstar}{$\mathrm{C}^*$}

\newcommand{\cU}{\mathcal{U}}
\newcommand{\cZ}{\mathcal{Z}}
\newcommand{\cF}{\mathcal{F}}
\newcommand{\cP}{\mathcal{P}}

\newcommand{\cB}{\mathcal{B}}
\newcommand{\cK}{\mathcal{K}}

\newtheorem{theorem}{Theorem}[section]
\newtheorem*{theorem*}{Theorem}
\newtheorem{proposition}[theorem]{Proposition}
\newtheorem{problem}[theorem]{Problem}
\newtheorem*{proposition*}{Proposition}
\newtheorem{lemma}[theorem]{Lemma}
\newtheorem*{lemma*}{Lemma}
\newtheorem{corollary}[theorem]{Corollary}
\newtheorem*{corollary*}{Corollar}

\newtheorem*{fact*}{Fact}
\theoremstyle{definition}
\newtheorem{definition}[theorem]{Definition}
\newtheorem*{definition*}{Definition}
\newtheorem{claim}[theorem]{Claim}
\newtheorem*{claim*}{Claim}

\newtheorem*{conjecture*}{Conjecture}

\theoremstyle{remark}
\newtheorem{example}[theorem]{Example}
\newtheorem*{example*}{Example}
\newtheorem{remark}[theorem]{Remark}
\newtheorem*{remark*}{Remark}

\newtheorem*{note*}{Note}
\newtheorem*{question*}{Question}

\newcommand{\norm}[1]{\left\lVert #1 \right\rVert}

\DeclareMathOperator{\Span}{span}

\DeclareMathOperator{\supp}{supp}

\DeclareMathOperator{\diam}{diam}

\DeclareMathOperator{\Proj}{Pr}

\newcommand{\cM}{\mathcal M} 
\newcommand{\cN}{\mathcal N}

\newcounter{my_enumerate_counter}
\newcommand{\pushcounter}{\setcounter{my_enumerate_counter}{\value{enumi}}}
\newcommand{\popcounter}{\setcounter{enumi}{\value{my_enumerate_counter}}}

\usepackage{enumitem}

\title{A quantization of coarse spaces and uniform Roe algebras}

\author[B. M. Braga]{Bruno M. Braga}

\address[B. M. Braga]{IMPA, Estrada Dona Castorina 110, 22460-320, Rio de Janeiro,
Brazil}
\email{demendoncabraga@gmail.com}
\thanks{B. M. Braga and D. Sherman were partially supported by NSF Grant   DMS 2054860. B. M. Braga  was partially supported by FAPERJ (Proc. E-26/200.167/2023) and by CNPq (Proc. 303571/2022-5)}

\author[J. Eisner]{Joseph Eisner}
\address[J. Eisner]{}
\email{josepheisner54@gmail.com}

\author[D. Sherman]{David Sherman}
\address[D. Sherman]{University of Virginia, $141$ Cabell Drive, Kerchof Hall, P.O. Box $400137$, Charlottesville, VA $22904$-$4137$,  USA } \email{dsherman@virginia.edu}

\begin{document}

 \begin{abstract}
  We propose a quantization of coarse spaces and uniform Roe algebras.  The objects are based on the quantum relations introduced by N. Weaver  and require the choice of a represented von Neumann algebra.  In the case of the diagonal inclusion $\ell_\infty(X) \subset  \cB(\ell_2(X))$, they reduce to the usual constructions.  Quantum metric spaces furnish natural examples parallel to the classical setting, but we provide other examples that are not inspired by metric considerations, including the new class of support expansion \cstar-algebras.  We also work out the basic theory for maps between quantum coarse spaces and their consequences for quantum uniform Roe algebras. 
 \end{abstract}

\maketitle


\section{Introduction}\label{SectionIntro}

 Roe-type algebras, also known as translation invariant algebras \cite{RoeBook}, are operator algebras built out of metric or more general coarse spaces.  Originally introduced by  J. Roe   \cite{Roe1988}    to obtain index theorems for elliptic operators on non-compact Riemannian manifolds, they have since found applications in many directions, from the  Baum-Connes and Novikov conjectures \cite{Yu2000} to topological insulators \cite{EwertMeyer2019}.  Based on a wave of recent work, we now know that in many situations these algebras are complete invariants for the large-scale, or \emph{coarse}, geometry of the underlying spaces (e.g., \cite{BBrigidity,MartinezVigolo2025}).  In this article we point out that the entire framework implicitly relies on the simplest von Neumann algebras, those of the form $\ell_\infty(X)$, and we generalize to arbitrary von Neumann algebras.  This opens up a new realm of quantum coarse spaces and their associated quantum uniform Roe algebras.  Here we follow standard usage of the adjective ``quantum" (much repeated throughout the paper) as an interpretation of structures connected to Hilbert spaces as noncommutative versions of classical counterparts \cite{Connes1994NoncommutativeGeo,WeaverBook2001Quantum}. The foundation for our approach is N. Weaver's quantization of relations on a set \cite{Weaver2012MemoirAMS}, as well as G. Kuperberg  and N. Weaver's quantum approach to metric spaces   \cite{KuperbergWeaver2012MemAMS}.  

Let us review a few concepts from coarse geometry.  A \emph{coarse space} consists of a set $X$ together with a collection $\cE$ of relations which, by satisfying certain axioms, gives a notion of boundedness.  The prototype is a metric space $(X,d)$, with $\cE$ comprising all subsets of the sets $\{(x,y) \in X \times X \mid \: d(x,y) \leq r\}$; the appropriate well-behaved maps in this setting are those for which the expansion of distances is controlled at the large scale.  Importantly, though, the axioms for coarse spaces 
allow for non-metric examples.  In fact coarse spaces are conceptually analogous to the much older framework of \emph{uniform spaces} \cite[Chapter 8]{Engelking}, in which uniform continuity (controlled expansion of distances at the small scale) is generalized beyond the metric setting.  Given a coarse space $(X, \cE)$, one may construct its \emph{uniform Roe algebra} $\cstu(X, \cE)$ as follows.  For any relation $E$ on $X$, i.e., $E \subset  X \times X$, we say that an operator $a \in \cB(\ell_2(X))$ with matrix form $[a_{xy}]_{x,y\in X}$ is \emph{$E$-controlled} if $\{(x,y)\in X\times X: \,  a_{xy}\neq 0\} \subset  E$.  Then $\cstu(X, \cE)$ is the unital \cstar-subalgebra of $\cB(\ell_2(X))$ obtained as the norm closure of all operators controlled by a member of $\cE$.

Our initial motivation is to propose and investigate analogues of the preceding paragraph in which relations on a set are replaced with quantum relations.  Classical uniform Roe algebras are already noncommutative and tied to Hilbert space, so how can they be quantized?  The answer is that quantum relations rely on an underlying represented von Neumann algebra, and quantum relations on the diagonal abelian subalgebra $\ell_\infty(X) \subset \cB(\ell_2(X))$ are in one-to-one correspondence with classical relations on $X$.  Here we define a \emph{quantum coarse space}, which features an appropriate collection of quantum relations, and we explain how to produce its \emph{quantum uniform Roe algebra}.  In case the underlying von Neumann algebra is $\ell_\infty(X) \subset \cB(\ell_2(X))$, this recovers the classical theory, but there is no need for it to be commutative or atomic.  And just as in the classical setting, one gets examples from the canonical quantum coarse structure of a quantum metric space \cite{KuperbergWeaver2012MemAMS}, but also as before there are non-metric examples of interest.  We develop here a new class of \cstar-algebras, \emph{support expansion \cstar-algebras}, that arise as (non-metric) quantum uniform Roe algebras.  We also put in due diligence to show that many basic concepts and facts for maps between coarse spaces have satisfactory quantum analogues.  


\smallskip

We now give an overview of the paper.  Let $X$ be an arbitrary set, and fix a Hilbert space $H$ and a von Neumann algebra  $\cM\subset  \cB(H)$.

In Section 2 we review the basic theory of quantum relations developed by N. Weaver.  A \emph{quantum relation on $\cM \subset  \cB(H)$} is  a weak$^*$-closed subspace $\cV\subset  \cB(H)$ which is also a bimodule over $\cM'$, i.e., $\cM'\cV\cM'\subset  \cV$ (see Definition \ref{DefiQuantumRelOriginal}). Such an object can also be encoded by a family of pairs of projections from $\cM\bar\otimes \cB(\ell_2)$ satisfying certain axioms; in this form, for which no representation of $\cM$ is required, it is called an \emph{intrinsic quantum relation on $\cM$} (see Definition \ref{DefiQuantumRel}).  In the case of a diagonal inclusion $\ell_\infty(X) \subset \cB(\ell_2(X))$, a third way of encoding the same information is as a subset $E \subset  X \times X$ (\cite[Proposition 2.2]{Weaver2012MemoirAMS} and Proposition \ref{Weaveratomic}), i.e., a relation.  The quantum relation corresponding to $E$ is nothing but the set of $E$-controlled operators.

In  Section \ref{SectionQCoarseSp} we define a \emph{quantum coarse structure on $\cM \subset  \cB(H)$} as a family $\mathscr{V}$ of quantum relations on $\cM \subset  \cB(H)$ satisfying some properties which mimic the standard axioms of coarse spaces  (Definition \ref{DefiQuantumCoarseSpaceOriginal}) and, analogously, we define  an  \emph{intrinsic quantum coarse structure on $\cM$}  as a family $\mathscr{R}$ of intrinsic quantum relations  on $\cM $ satisfying some similar  properties (Definition \ref{DefiQuantumCoarseSpace}).  The pairs $(\cM \subset  \cB(H), \mathscr{V})$ and $(\cM,\mathscr{R})$ are then called a \emph{quantum coarse space} and an \emph{intrinsic quantum coarse space}, respectively. These  notions are canonically equivalent to each other (Corollary \ref{CorQRelAndIQRel}) and coincide in a canonical way  with classical coarse structures in the commutative counting measure case $\cM=\ell_\infty(X)$.   We also show that, just as in classical coarse geometry,  a quantum coarse structure is \emph{(quantum) metrizable} if and only if it is  countably generated (Proposition \ref{PropMetrizable}).   Immediate examples of quantum coarse spaces come from quantum graphs and, more generally, quantum metric spaces \cite{KuperbergWeaver2012MemAMS}.

Section \ref{SectionQURA} introduces \emph{quantum uniform Roe algebras} (Definition \ref{DefiQuantumURoeAlgGEN}), which are simply the unital \cstar-algebras that arise as the closed union of all the quantum relations in a quantum coarse space. 
In the case $\ell_\infty(X) \subset \cB(\ell_2(X))$, the reader will notice that this coincides with the usual uniform Roe algebra.  Section \ref{basicprop} describes quantum versions of connectedness and triviality (for which we amusingly find three distinct levels), and Section \ref{examples} presents some basic examples.
\emph{Intrinsic} quantum relations often allow for a   more intuitive approach to quantum large scale geometry, but  
the construction of a quantum uniform Roe algebra requires that $\cM$ be represented.  This does not muddy the waters too much, since the representation theory of von Neumann algebras is simple.  Given an intrinsic quantum coarse structure on $\cM$, the associated quantum uniform Roe algebra is determined up to a ``change of representation" equivalence relation (Section \ref{subsectioniqura}; see also Theorem \ref{ThmEmbAndIso}(2)).

Section \ref{SubsectionDimFunc} presents a new class of examples that illustrate the flexibility of our definitions and invite further study.  The prototype of a classical coarse structure arises from a metric; the uniform Roe algebra is the closure of the operators that do not change the support of a vector too much in terms of \emph{displacement}.  Analogously: if $\cM \subset  \cB(H)$ is equipped with a faithful normal semifinite trace $\tau$, we construct in Section \ref{SubsectionExamDimFun} a \cstar-algebra as the closure of the operators in $\cB(H)$ that do not change the size of subspaces affiliated with $\cM$ too much in terms of \emph{measure}, where the size of a subspace is measured by applying the trace to the associated projection in $\cM$.  The explicit condition on $a \in \cB(H)$ is the existence of $\lambda \geq 0$ such that for all projections $q \in \cM$,
\[\tau(s_\ell^\cM(a q)), \: \tau(s_\ell^\cM(a^*q)) \leq \lambda \cdot \tau(q).\]
(Here $s_\ell^\cM(\cdot)$ denotes the \emph{left $\cM$-support}, the smallest projection in $\cM$ fixing the range of the operator.)  We call this \cstar-algebra a \emph{support expansion \cstar-algebra} and show in Theorem \ref{thmprojsuppexp} that it is a quantum uniform Roe algebra.

In Section \ref{SubsectionAbelianMeasDistorsion} we give a vector-based version of this construction.  It is in some ways simpler and produces the same \cstar-algebra when $\cM$ is abelian (Theorem \ref{abelianexpansion}), but in general a \emph{vector support expansion \cstar-algebra} need not be a quantum uniform Roe algebra at all (Theorem \ref{notqura}).  In our companion paper \cite{EisnerThesis} we study the wild jungle of (vector) support expansion \cstar-algebras arising when $\cM$ is restricted to be abelian but the constraint functions $f$ --- meaning $\tau(s_\ell^\cM(aq)) \leq f(\tau(q))$ --- are not necessarily linear.  Subsection \ref{SubsectionII1Factor} makes the observation that all *-isomorphisms between support expansion \cstar-algebras associated to $\mathrm{II}_1$-factors without property $\Gamma$ are spatially implemented (Proposition \ref{ThmII1notGammaIsoUnitaryImpl}), and we explain how this could be a step toward rigidity-type results for quantum uniform Roe algebras that may be pursued elsewhere.

Section \ref{SectionQuantumCoarseEq} deals with morphisms,  equivalences, and embeddings  between quantum coarse spaces --- here intrinsic quantum relations provide a more suitable framework. We proceed by imposing various conditions on unital weak$^*$-continuous $^*$-homomorphisms $\cM\to \cN$ (``quantum functions") so that they interact appropriately with intrinsic quantum coarse structures.  
This leads to quantum versions for the following terms from coarse geometry: coarse function, coarse isomorphism, close functions, coarse equivalence, coarse subspace, coarse embedding, expanding, cobounded.  Of course, we are especially interested in their consequences for quantum uniform Roe algebras. Looking back to the classical scenario once again, we know that if there is an injective coarse map $f:X\to Y$, then there is a canonical embedding $\cstu(X)\hookrightarrow \cstu(Y)$. If  $f$ is furthermore a coarse embedding, then the image of the embedding $\cstu(X)\hookrightarrow  \cstu(Y)$ is a hereditary subalgebra of $\cstu(Y)$; if $f$ is a bijective coarse equivalence, the embedding $\cstu(X)\hookrightarrow  \cstu(Y)$ becomes an isomorphism. These results have quantum analogs, and we prove them in Theorem \ref{ThmEmbAndIso}(2), Theorem \ref{ThmEmbAndIsoF}, and Theorem \ref{ThmEmbAndIsoHERE}.

Section \ref{SectionMetricModuli} recasts some of our definitions in terms of natural moduli for quantum functions between quantum metric spaces (Propositions \ref{PropModulusUniformCont} and \ref{PropModulusExp}). In Section \ref{SubsectionAsyDim} we define the asymptotic dimension of a (not necessarily metrizable) quantum coarse space and show that, as in the classical setting, this notion  is stable under quantum coarse embedding (Theorem \ref{ThmAsyDim}).

\smallskip

This paper is mostly an attempt to lay the groundwork for quantum coarse geometry (although we believe support expansion \cstar-algebras are natural and independently interesting operator algebras).  We are hopeful that many compelling examples and phenomena are yet to be discovered.

\section{Preliminaries: quantum relations}\label{SectionPrelim}
\subsection{Basic notation}
Given a Hilbert space $H$, we denote the \cstar-algebra of all bounded operators on $H$ by $\cB(H)$ and the ideal of compact operators on $H$ by $\cK(H)$.     We denote the identity element by $1$, occasionally with a subscript to indicate the scope, e.g., $1_H$.  Given a \cstar-algebra $\cA$, we denote the set of its projections by $\mathrm{Pr}(\cA)$ and the subset of its nonzero projections by $\mathrm{Pr}_*(\cA)$.  

A measure  space $(X,\mu)$ is called \emph{finitely decomposable} if $X$ has a  partition  into finite measure subsets, say $X=\bigsqcup_{\lambda\in \Lambda} X_\lambda$, so that  $A\subset  X$ is measurable if and only if each $A\cap X_\lambda$ is measurable, and in this case $\mu(A)=\sum_{\lambda\in \Lambda}\mu(A\cap X_\lambda)$.  We frequently identify $f \in L_\infty(X,\mu)$ with the corresponding operator of multiplication by $f$, so that $L_\infty(X,\mu)\subset  \cB(L_2(X,\mu))$. Given a measurable $A\subset  X$, $\chi_A$ denotes the characteristic function of $A$, which is a projection in $\cB(L_2(X,\mu))$.  In case $\mu$ is the counting measure on $X$, we denote the standard unit basis of $\ell_2(X)$ by $(\delta_x)_{x\in X}$ and, given $x,y\in X$,  $e_{xy}$ denotes the rank 1 partial isometry on $\ell_2(X)$ which takes $\delta_y$ to $\delta_x$; so, $\chi_{\{x\}}=e_{xx}$.

For a represented von Neumann algebra $\cM \subset  \cB(H)$, we define the \textit{$\cM$-support} of a vector $\xi \in H$ as the smallest projection in $\cM$ fixing $\xi$, denoted $s^\cM(\xi)$.  For $a \in \cB(H)$ its \textit{left $\cM$-support}, written $s^\cM_\ell(a)$, is the smallest projection $q\in \cM$ with $qa=a$.  This generalizes the range projection ($\cM = \cB(H)$) and is equal to ${\vee}_{\xi \in H} s^\cM(a\xi)$.  Here are some elementary properties we will use in the sequel, where the sums are convergent in any sense:
\begin{equation} \label{E:suppproperties}
s^\cM\left(\sum \xi_j\right) \leq \bigvee s^\cM(\xi_j); \qquad s^\cM_\ell \left(\sum a_j\right) \leq \bigvee s^\cM_\ell(a_j).
\end{equation}

 \subsection{Quantum relations}
We recall N. Weaver's  \emph{quantum relations}:

\begin{definition}\label{DefiQuantumRelOriginal}
Let $\cM\subset  \cB(H)$ be a von Neumann algebra. A weak$^*$-closed $\cM'$--$\cM'$ bimodule $\cV\subset  \cB(H)$ is called a \emph{quantum relation on $\cM$}. We denote the set of all quantum relations on $\cM \subset  \cB(H)$ by $\mathrm{QRel}(\cM \subset  \cB(H))$.
\end{definition}

The next proposition justifies why the objects introduced in Definition \ref{DefiQuantumRelOriginal} deserve to be called quantum relations: for a set $X$, quantum relations on the diagonal $\ell_\infty(X) \subset  \cB(\ell_2(X))$ correspond canonically to relations on $X$.

\begin{proposition} \label{Weaveratomic}\emph{(}\cite[Proposition 2.2]{Weaver2012MemoirAMS}\emph{).}
Let $X$ be a set and consider the von Neumann algebra $\ell_\infty(X)\subset  \cB(\ell_2(X))$. If $E$ is a relation on $X$, then 
\[\cV_E=\Big\{a\in \cB(\ell_2(X)): \,   (x,y)\not\in E\text{ implies } \chi_{\{x\}} a\chi_{\{y\}}=0\Big\}\]
is a quantum relation on $\ell_\infty(X)$.  If $\cV$ is  a quantum relation on $\ell_\infty(X)$, then 
\[E_\cV=\Big\{(x,y)\in X^2: \,  \exists a\in \cV\text{ for which }\chi_{\{x\}} a\chi_{\{y\}}\neq 0\Big\}\]
is a relation on $X$. These constructions are inverse to each other. \label{PropositionWeaverRelation}
\end{proposition}
In other words, $E$ is the collection of matrix entries where elements of $\cV_E$ are allowed to be nonzero.  We say that operators in  $\cV_E$  are   \textit{controlled by $E$} or have \emph{support controlled by $E$}.

N. Weaver also gave an ``intrinsic" approach in which a quantum relation on $\cM \subset  \cB(H)$ corresponds to a family of pairs of projections in $\cM \bar \otimes \cB(\ell_2)$.  This description does not require that $\cM$ act on a Hilbert space.

\begin{definition}\label{DefiQuantumRel}
Let  $\cM$ be a von Neumann algebra and consider $\cP=\Proj(\cM\bar{\otimes}\cB(\ell_2))$ endowed with the restriction of the weak operator topology.  An   open subset $\cR\subset \cP\times \cP$  is called an \emph{intrinsic quantum relation on $\cM$} if the following hold.
\begin{enumerate}
\item\label{ItemDefiQuantumRel1} $(0,0)\not\in \cR$.
\item\label{ItemDefiQuantumRel2} Given  families of nonzero projections $(p_i)_{i\in I}$ and $(q_j)_{j\in J}$ in $\cP$, we have  
\[\Big(\bigvee_{i\in I} p_i, \bigvee_{j\in J} q_j\Big)\in \cR\ \iff \ \exists (i,j)\in I\times J\ \text{ with }\ (p_i,q_j)\in \cR.\]
\item \label{ItemDefiQuantumRel3} For all projections $p,q\in \cP$ and all $b\in 1_{\cM}\otimes \cB(\ell_2)$, we have 
\[(p,[bq])\in \cR\ \iff\ ([b^*p],q)\in \cR.\]
Here square brackets
denote range projection.
\end{enumerate}
  We denote the set of all  intrinsic quantum relations on $\cM$  by $\mathrm{IQRel}(\cM)$.
\end{definition}
 The correspondence between $\mathrm{IQRel}(\cM)$ and $\mathrm{QRel}(\cM \subset  \cB(H))$ is described in the following theorem.  Informally, in an associated intrinsic quantum relation the pairs $(p,q)$ describe corners where arrays for amplifications of operators in the quantum relation are sometimes nonzero.

 \begin{theorem}\emph{(}\cite[Theorem 2.32]{Weaver2012MemoirAMS}\emph{).}
 Let $\cM\subset  \cB(H)$ be a von Neumann algebra and $\cP=\Proj(\cM\bar\otimes \cB(\ell_2))$. If $\cV$ is a quantum relation on $\cM$, then 
 \[\cR_\cV=\Big\{(p,q)\in \cP^2: \,  \exists a\in \cV\text{ with }p(a\otimes 1)q\neq 0\Big\}\] 
 is an intrinsic quantum relation on $\cM$. If $\cR$ is an intrinsic quantum relation on $\cM$, then 
 \[\cV_\cR=\Big\{a\in \cB(H): \,  (p,q)\not\in \cR\Rightarrow p(a\otimes 1)q=0\Big\}\]
 is a quantum relation on $\cM \subset  \cB(H)$. These  constructions are inverse to each other. \label{ThmWeaverQRelAndIntQRel}
 \end{theorem}
 
We make extensive  use of the notations $\cV_\cR$ and $\cR_\cV$ of   Theorem \ref{ThmWeaverQRelAndIntQRel} throughout the paper. 

Given two faithful representations $\pi_j: \cM \to \cB(H_j)$, $j=1,2$, Theorem \ref{ThmWeaverQRelAndIntQRel} tells us that each $\mathrm{QRel}(\pi_j(\cM) \subset  \cB(H_j))$ is in correspondence with $\mathrm{IQRel}(\cM)$ and thus with each other.  This correspondence is discussed in \cite[Theorem 2.7]{Weaver2012MemoirAMS}, and for our use in the sequel we make it entirely explicit here.  Since any isomorphism between represented von Neumann algebras can be decomposed into an amplification, a spatial isomorphism, and a reduction \cite[Theorem IV.5.5]{Takesaki2002BookI}, it suffices to give the quantum relation corresponding to $\cV \in \mathrm{QRel}(\cM \subset  \cB(H))$ under each of these three types of maps, which is as follows.
\begin{itemize}
    \item Amplification by a Hilbert space $K$:\\ $\cV \bar{\otimes} \cB(K) \in \mathrm{QRel}(\cM \otimes 1_K \subset  \cB(H) \bar{\otimes} \cB(K) \simeq \cB(H \otimes K))$.  Here we are using the normal spatial tensor product for $\cV \bar{\otimes} \cB(K)$, meaning the weak* closure of the algebraic tensor product $\cV \odot \cB(K)$ inside $\cB(H \otimes K)$.  Note that $\cV$ is not amplified to $\cV \otimes 1_K$ but replaced with the ``larger" $\cV \bar{\otimes} \cB(K)$.
    \item Spatial isomorphism via a unitary $u$ from $H$ to $H'$:\\ $u\cV u^* \in \mathrm{QRel}(u \cM u^* \subset  \cB(H'))$.
    \item Reduction by a projection $p' \in \cM'$ with full central support in $\cM'$:\\ $(p'\cV) |_{p'H} \in \mathrm{QRel}(p'\cM \subset  \cB(p'H))$.\footnote{Recall that the central support of a projection in a von Neumann algebra is the smallest central projection dominating that projection.}
\end{itemize}
For the last, the central support condition implies that $p'\cM \simeq \cM$.

The \textit{diagonal} intrinsic quantum relation $\Delta_{\cM}$ on $\cM$ is defined by $(p,q) \in \Delta_{\cM} \iff pq \neq 0$.  It is easy to check that when $\cM \subset  \cB(H)$, $\cV_{\Delta_\cM} = \cM'$.  We sometimes write just $\Delta$ if the von Neumann algebra is implicit from the context. 
 
\subsection{Operator reflexivity}\label{SubsecRefle}  From $\cV = \cV_{\cR_\cV}$ in Theorem \ref{ThmWeaverQRelAndIntQRel}, we know that there are sufficiently many projections in $\cP=\Proj(\cM\bar\otimes \cB(\ell_2))$ to determine $\cV$.  Sometimes there are sufficiently many projections already in the $1\times 1$ level, $\text{Pr}(\cM)$.  We follow N. Weaver \cite[Section 2.5]{Weaver2012MemoirAMS} in the definitions below. 

\begin{definition}
For any subset  $\cV\subset  \cB(H)$, the \emph{operator reflexive closure of $\cV$} is defined by 
\[\text{orc}(\cV)=\{a\in \cB(H): \,  \forall p,q\in \Proj(\cB(H)),\ p\cV q=0\Rightarrow paq=0\}.\]
This is always a $w^*$-closed linear subspace of $\cB(H)$.  We say that $\cV$ is \emph{operator reflexive} if $\cV=\text{orc}(\cV)$.
\end{definition}

It is easy to see that operator reflexive spaces are closed under intersection and operator adjoint, and that $\text{orc}(\cV)$ is the smallest operator reflexive space containing $\cV$.

If $\cV$ is an $\cM'$-bimodule, then $\text{orc}(\cV)$ is a quantum relation over $\cM$, and one only needs projections in $\cM$ to define it (see \cite[Propositions 2.15, 2.18]{Weaver2012MemoirAMS}):
\[\text{orc}(\cV)=\{a\in \cB(H): \, \forall p,q\in \Proj(\cM),\  p\cV q=0\Rightarrow paq=0\}.\]

Quantum relations over atomic abelian von Neumann algebras are always operator reflexive.  Here is a simple non-example, with supporting linear algebra details left to the reader: identifying $\cB(\mathbb{C}^2)$ with $\mathbb{M}_2$, the subspace $\{ ( \begin{smallmatrix} a & b \\ c & a \end{smallmatrix} ) \mid a,b,c\in \C\}$ is a quantum relation over $\cB(\mathbb{C}^2) \subset  \cB(\mathbb{C}^2)$ that is not operator reflexive.   Thanks to N. Weaver for originally suggesting this example.

\begin{remark}
Here is the origin of the terminology.  A unital operator algebra $\cA \subset  \cB(H)$ is classically said to be ``reflexive" if no operator outside $\cA$ preserves all the invariant subspaces of $\cA$.  In fact, this is equivalent to operator reflexivity of $\cA$ as defined above (\cite[Proposition 2.19]{Weaver2012MemoirAMS}); the longer term is used in \cite{Weaver2012MemoirAMS} because ``reflexive" already has a meaning for a (quantum) relation, namely, that it contains the diagonal.  
\end{remark}

\begin{example} \label{oprefex} (cf.\ \cite{Erdos1986})
Let $\cM \subset  \cB(H)$.  Let $\varphi$ be any map from $\Proj(\cM)$ to $\Proj(\cM)$, and define $$\cV_\varphi = \{a \in \cB(H) : \,  s_\ell^\cM(aq) \leq \varphi(q), \, \forall q \in \Proj(\cM)\}.$$
Note that
\begin{equation} \label{phichar}
s_\ell^\cM(aq) \leq \varphi(q) \; \iff \; \varphi(q)aq = aq \; \iff \; \varphi(q)^\perp aq = 0.
\end{equation}
Then $\cV_\varphi$ is an operator reflexive quantum relation on $\cM$.  That it is a weak*-closed $\cM'$-bimodule is perhaps easiest to see from the last condition in \eqref{phichar}.  For operator reflexivity, suppose that $a \in \cB(H)$ has the property that $p\cV_\varphi q=0 \Rightarrow paq =0$ for all $p,q\in \Proj(\cM)$.  Since $\varphi(q)^\perp \cV_\varphi q = 0$ for all $q\in \Proj(\cM)$, we have that $\varphi(q)^\perp aq=0$ as well, and then $a \in \cV_\varphi$.
\end{example}

Example \ref{oprefex} actually characterizes operator reflexive relations, as we see in the following proposition.

\begin{proposition}
Keep the notations of Example \ref{oprefex}.  For any quantum relation $\cV$ on $\cM \subset \cB(H)$ we have $\text{orc}(\cV) = \cV_{\varphi_\cV}$, where $\varphi_\cV: \mathrm{Pr}(\cM)\to \mathrm{Pr}(\cM)$ is defined by $ \varphi_\cV(q) = \vee_{a \in \cV} s_\ell^\cM(a q)$.  Thus the operator reflexive relations on $\cM\subset \cB(H)$ are exactly those of the form $\cV_\varphi$.
\end{proposition}

\begin{proof}
From Example \ref{oprefex} we know $\cV_{\varphi_\cV}$ is an operator reflexive relation.  It contains $\cV$: if $b \in \cV$ and $q \in \mathrm{Pr}(\cM)$, $s_\ell^\cM(bq) \leq \vee_{a \in \cV} s_\ell^\cM(a q) = \varphi_\cV(q)$.  This shows $\text{orc}(\cV) \subset \cV_{\varphi_\cV}$.

For the opposite inclusion, suppose $p,q \in \mathrm{Pr}(\cM)$ are such that $p\cV q = 0$.  For any $a \in \cV$, $paq = 0$ and so $p s_\ell^\cM(aq) = 0$.  It follows that $p \varphi_\cV(q) =  p \vee_{a \in \cV} s_\ell^\cM(a q) = 0$.
If $c \in \cV_{\varphi_\cV}$, then $\varphi_\cV(q) \geq s_\ell^\cM(cq)$, so by the foregoing
$p s_\ell^\cM(cq) = 0$ and thus $pcq =0$.  This gives $c \in \text{orc}(\cV)$.
\end{proof}

\begin{remark}
Operator reflexivity relies on the relation between  $\textrm{Pr}(\cM)^2$ and $\cB(H)$ consisting of the set $\{((p,q), a) \mid \, paq = 0\}$: $\text{orc}(\cV)$ is a ``double-perp" or ``double-commutant" type of closure.  There is a dual notion of closure in $\textrm{Pr}(\cM)^2$, or equivalently in the space of maps $\mathrm{Pr}(\cM)\to \mathrm{Pr}(\cM)$: in the language above, the ``closure" of $\varphi: \mathrm{Pr}(\cM)\to \mathrm{Pr}(\cM)$ is  $\varphi_{\cV_\varphi}$, which is $ \leq \varphi$.  This perspective is discussed more in \cite[Section 4.3]{EisnerDissertation}.
\end{remark}

\section{Quantum coarse spaces}  
\label{SectionQCoarseSp}

The notion of quantum relations  provides us with an appropriate framework to define quantum coarse structures on von Neumann algebras, which in turn will allow us to define quantum uniform Roe algebras in Section \ref{SectionQURA}. In this section, we introduce quantum and intrinsic quantum coarse spaces, notice that they are equivalent in a canonical way, and discuss their metrizability.  We postpone the investigation of morphisms and equivalences in the category of quantum coarse spaces to Section \ref{SectionQuantumCoarseEq}.

\subsection{Quantum coarse spaces}\label{SubsectionQCoarseSp}
 Recall that a \emph{coarse space} is a set $X$ together with a \emph{coarse structure} $\cE\subset  \cP(X^2)$ on $X$, i.e., $\cE$ is a family of subsets of $X^2$ containing the  diagonal $\Delta_X=\{(x,x): \,  x\in X\}$ and which  is closed under subsets, inverses,\footnote{If $E\subset  X^2$, then the \emph{inverse of $E$} is $E^{-1}=\{(x,y): \,  (y,x)\in E\}$.}  finite unions, and compositions\footnote{If $E,F\subset  X^2$, their \emph{composition} is given by $E\circ F=\{(x,y): \,  \exists z\in X, \ (x,z)\in E\text{ and }(z,y)\in F\}$.} (see \cite{RoeBook} for a detailed monograph on coarse spaces). The elements of $\cE$ are nothing but relations on $X$, often called \emph{entourages} in this context.
 
 This can be generalized to the quantum world as follows.

\begin{definition}\label{DefiQuantumCoarseSpaceOriginal}
Let $\cM\subset  \cB(H)$ be a von Neumann algebra. A family $\mathscr{V}$ of quantum relations on $\cM \subset  \cB(H)$ is called a \emph{quantum coarse structure on $\cM \subset  \cB(H)$} if
\begin{enumerate}
\item\label{ItemDefiQuanCoarseDIAG} $\cM'\in \mathscr{V}$,
\item\label{ItemDefiQuanCoarseSUB} $  \cV_1\in \mathscr{V}$,  $\cV_2\in \mathrm{QRel}(\cM\subset  \cB(H))$, and $\cV_2\subset  \cV_1$ imply  $\cV_2\in \mathscr{V}$,
\item \label{ItemDefiQuanCoarseINV} $\cV_1\in \mathscr{V}$ implies $\cV_1^*\in \mathscr{V}$ (here $\cV_1^*$ indicates the set of adjoints of elements of $\cV_1$),
\item \label{ItemDefiQuanCoarseCUP} $\cV_1,\cV_2\in \mathscr{V}$ implies $\overline{\cV_1+\cV_2}^{w^*}\in \mathscr{V}$, and
\item\label{ItemDefiQuanCoarseCOMP} $\cV_1,\cV_2\in \mathscr{V}$ implies $ \overline{\Span}^{w^*}(\cV_1 \cV_2)\in \mathscr{V}$.
\end{enumerate}
The pair $(\cM \subset  \cB(H),\mathscr{V})$ is called a \emph{quantum coarse space}.
\end{definition}

 Notice that items \eqref{ItemDefiQuanCoarseDIAG}, \eqref{ItemDefiQuanCoarseSUB},  \eqref{ItemDefiQuanCoarseINV},  \eqref{ItemDefiQuanCoarseCUP}, and \eqref{ItemDefiQuanCoarseCOMP} of Definition \ref{DefiQuantumCoarseSpaceOriginal} are   the quantum versions of a coarse structure  $\cE$ containing $\Delta_X$ and being closed under subsets, inverses,  finite unions, and  compositions, respectively. Moreover,  if $\cM=\ell_\infty(X)$ and $H=\ell_2(X)$ for some set $X$,   Proposition \ref{PropositionWeaverRelation} and \cite[Proposition 2.5]{Weaver2012MemoirAMS} provide a canonical  equivalence between classical coarse structures $\cE$ on $X$ and quantum coarse structures $\mathscr{V}_\cE=\{\cV_E: \,  E\in \cE\}$ on $\ell_\infty(X)$.

 \begin{definition}
 We will say that a quantum coarse structure $\mathscr{V}$, or a quantum coarse space $(\cM \subset  \cB(H),\mathscr{V})$, is \emph{operator reflexive} if $\text{orc}(\cV) \in \mathscr{V}$ for all $\cV\in \mathscr{V}$.
 \end{definition}
 
If $\mathbf{U}$ 
is a family of quantum relations on a von Neumann algebra $\cM \subset  \cB(H)$, then the \emph{quantum coarse structure generated by $\mathbf{U}$}, denoted by $\mathscr{V}_\mathbf{U}$, is the   intersection of all quantum coarse structures which contain $\mathbf{U}$. 
If $\mathbf{U}$ consists of the single quantum relation $\cU$, we may simply write $\mathscr{V}_{\cU}$ instead of $\mathscr{V}_\mathbf{U}$.

\subsection{Metrizability of quantum coarse spaces}

\begin{definition} (\cite[Definition 2.3]{KuperbergWeaver2012MemAMS}).
Let $\cM\subset  \cB(H)$ be a von Neumann algebra. A family $\mathbf{V}=(\cV_t)_{t\geq 0}$ of weak$^*$-closed operator systems\footnote{A closed subspace of $\cB(H)$ is an \emph{operator system} if it is self-adjoint and contains $1_H$.} in $\cB(H)$ is called a \emph{quantum metric on $\cM \subset  \cB(H)$} if
\begin{enumerate}
    \item $\cV_0=\cM'$,
    \item $\cV_t\cV_s\subset  \cV_{t+s}$ for all $t,s\geq 0$, and 
    \item $\cV_t=\bigcap_{s>t}\cV_s$ for all $t\geq 0$.
\end{enumerate}
The pair $(\cM\subset \cB(H),\mathbf{V}) $ is called a \emph{quantum metric space}.
\end{definition}
Notice that  the definition of quantum metric   implies that each $\cV_t$ is a quantum relation on $\cM$.  One thinks of $\cV_t$ as ``distance $\leq t$". The definition also implies that the quantum coarse structure   $\mathscr{V}_\mathbf{V}$ generated by $(\cM,(\cV_t)_{t\geq 0})$ is nothing but
\[\mathscr{V}_\mathbf{V}=\Big\{\cV\in \mathrm{QRel}(\cM\subset  \cB(H)): \,  \exists t\geq 0\text{ so that }\cV\subset  \cV_t\Big\}.\]
Clearly, $\mathscr{V}_\mathbf{V}$ is countably generated by $(\cM, (\cV_n)_{n=0}^\infty)$ (we discuss this further in Proposition \ref{PropMetrizable} below).

\begin{example} \label{quantummetric}
A metric space $(X,d)$ gives rise to a quantum metric on $X$ 
by setting \[\cV_t = \Big\{(a_{xy})_{x,y\in X}\in \cB(\ell_2(X)): \, d(x,y) > t  \Rightarrow a_{xy}=0\Big\},\]  a ``thickened diagonal" in $\cB(\ell_2(X))$.
\end{example}

\begin{example} \label{quantumgraph}
A quantum relation  $\cV\in\mathrm{QRel}(\cM \subset  \cB(H))$ is a \emph{quantum graph} if $\cM'\subset  \cV$ and $\cV=\cV^*$ (see \cite[Definition 2.6]{Weaver2012MemoirAMS}), i.e., if $\cV$ is a weak$^*$-closed operator system   which is a bimodule over $\cM'$. Then the quantum coarse structure generated by $\cV $, $\mathscr{V}_{\cV}$, is called the \emph{quantum graph coarse structure given by $\cV$}. It is nothing but the quantum metric in which $$\cV_t=\overline{\Span}^{w^*}(\underbrace{\cV\cdot\ldots\cdot \cV}_\text{$\lfloor t \rfloor$-times})$$
(here the 0-fold product is interpreted as the diagonal $\cM'$).

A classical graph $G$ with vertex set $X$ gives rise to a quantum graph on $\ell_\infty(X) \subset  \cB(\ell_2(X))$ by taking \[\cV = \{(a_{xy})_{x,y\in X}\in \cB(\ell_2(X)): \, \text{$x,y$ nonadjacent} \Rightarrow a_{xy}=0\}.\]  In other words, the corresponding relation on $X$ is ``adjacency in $G$," and the quantum graph coarse structure is the quantum metric associated to the path metric as in Example \ref{quantummetric}.
\end{example} 

In classical large scale geometry, it is well known that a  coarse space is metrizable if and only if its coarse structure is countably generated (\cite[Theorem 2.55]{RoeBook}). The next result shows that the same holds in the quantum world. We say that a quantum coarse structure $\mathscr{V}$ on a von Neumann algebra $\cM$ is \emph{metrizable} if $\mathscr{V}=\mathscr{V}_{\mathbf{V}}$ for some quantum metric $\mathbf{V}$ on $\cM$.

\begin{proposition}\label{PropMetrizable}
Let $(\cM,\mathscr{V})$ be a quantum coarse space. Then $\mathscr{V}$ is metrizable if and only if  it is countably generated.
\end{proposition}

\begin{proof}
 As noticed before Example \ref{quantummetric}, if $\mathscr{V}$ is metrizable then $\mathscr{V}$ is countably generated. Suppose then that $\mathscr{V}$  is the quantum coarse structure generated by a sequence of quantum relations  $(\cU_n)_{n=1}^\infty$ on $\cM$; replacing each $\cU_n$ by $\overline{\cU_{n}+\cU^*_n}^{w^*}$, if necessary, we can assume that each $\cU_n$ is an operator system.  We define $(\cV_n)_{n=0}^\infty$ inductively by setting $\cV_0=\cM'$  and, for $n\geq 0$, 
 \[\cV_{n+1}=\overline{ \Span }^{w^*}(\cV_{n}\cV_{n})+\cU_{n+1}.\]
 Since both $\cV_n$ and $\cU_{n+1}$ are closed under taking adjoints,   so is $\cV_{n+1}$. It is also clear that $\mathscr{V}$ is the quantum coarse structure generated by $(\cV_n)_{n=0}^\infty$, since the $\cV_n$ are in $\mathscr{V}$, and each $\cU_{n+1}$ is contained in $\cV_{n+1}$.
 
 Finally, since $\cV_n\cV_n\subset  \cV_{n+1}$ for all $n \geq 0$, it follows that  $\cV_{n}\cV_{m}\subset  \cV_{n+m}$ for all $n,m\in\N\cup\{0\}$. Letting $\cV_t=\cV_{\lfloor t\rfloor}$ for all $t\geq 0$, we obtain that $\mathbf{V}=(\cV_t)_{t\geq 0}$ is a quantum metric on $\cM$ which generates $\mathscr{V}$.
\end{proof}

\subsection{Intrinsic  quantum coarse spaces}
Intrinsic quantum relations  play a fundamental role in this paper, both in order to provide examples of quantum uniform Roe algebras (Section \ref{SubsectionDimFunc}) and to define morphisms between quantum coarse spaces (Section \ref{SectionQuantumCoarseEq}). We next introduce  the notion of intrinsic quantum coarse spaces, defined via the intrinsic versions of the axioms for quantum coarse spaces. Notice that  if  $\cR$ and $\cR'$ are intrinsic quantum relations, then  so are $\cR\cup\cR'$ and $\cR^{-1}=\{(p,q) : \,  (q,p)\in \cR\}$. As for the  composition of intrinsic quantum relations,  it follows from Theorem \ref{ThmWeaverQRelAndIntQRel} that  $\cR \circ \cR'$ should be defined as $\cR_{\cV_\cR \cV_{\cR'}}$\footnote{We allow ourselves the following abuse of notation: given quantum relations $\cV_1$ and $\cV_2$, we let   $\cR_{\cV_1\cV_2}=\cR_{\overline{\Span}^{w^*}(\cV_1\cV_2)}$ and $\cR_{\cV_1+\cV_2}=\cR_{\overline{\cV_1+\cV_2}^{w^*}}$. Notice that this is indeed an abuse of notation since  $\cV_1\cV_2$ and $\cV_1+\cV_2$ need not be quantum relations.}.   
We have the following description, inspired by \cite[Proposition 1.5]{Weaver2012MemoirAMS}.

\begin{proposition}\label{PropComposition}
Consider intrinsic quantum relations $\cR_1$ and $\cR_2$ on $\cM$. Then the composition
\[\cR_1\circ \cR_2=\Big\{(p,q)\in\cP^2 : \,    \forall r\in \cP,\   (p,r )\in \cR_1 \text{ or }
(r^\perp,q)\in \cR_2 \Big\}\]
is the intrinsic quantum relation corresponding to the quantum relations $\cV_{\cR_1} \cV_{\cR_2}$, i.e., $\cR_1\circ \cR_2=\cR_{\cV_1\cV_2}$. 
\end{proposition} 

\begin{proof}
To simplify notation, let $\cV_1=\cV_{\cR_1}$  and $\cV_1=\cV_{\cR_2}$. Notice that by Theorem  \ref{ThmWeaverQRelAndIntQRel} we have $\cR_i=\cR_{\cV_i}$ for $i\in \{1,2\}$. This will be used multiple times below.

Let $(p,q)\in  \cR_{\cV_1\cV_2}$. Then there must be $v_1\in \cV_1$ and $v_2\in \cV_2$ so that $p(v_1v_2\otimes 1)q\neq 0$. For any $r \in \cP$, we have
 \[ 0 \neq p(v_1 v_2 \otimes 1)q= p(v_1\otimes 1)r(v_2 \otimes 1)q + p(v_1\otimes 1)r^\perp(v_2 \otimes 1)q,\]
so that one of the summands is not zero.  Thus either   $(p,r )\in \cR_{1} \text{ or }
(r^\perp,q)\in \cR_{2}$.

 Now say $(p,q)\in \cR_1\circ \cR_2$. Let $K$ be the closure of all $\xi\in H\otimes \ell_2$ so that $p(v\otimes1)\xi =0$ for all $v\in \cV_1$. As $\cV_1$ is a bimodule over $\cM'$, it follows that $K$ is invariant under $\cM'\otimes 1$. So the projection onto $K$, say  $r$, must belong to $\cM\bar\otimes \cB(\ell_2)$. By the definition of $K$, we must have that  $p(v\otimes 1)r=0$ for all $v\in \cV_1$. Therefore, $(p,r)\not\in \cR_1$; by assumption we must have that $(r^\perp, q)\in \cR_2$. Hence there is $v_2 \in \cV_2$ so that $r^\perp(v_2 \otimes 1)q\neq 0$. By the definition of $K$ and $r$, there must be some $v_1\in \cV_1$ for which \[p(v_1\otimes 1)r^\perp(v_2\otimes 1)q \neq 0.\]
Since $p(v_1 \otimes 1)r=0$, we conclude that $p(v_1 v_2 \otimes 1)q\neq 0$; so  $(p,q)\in \cR_{\cV_1\cV_2}$. 
\end{proof}

We can now define \emph{intrinsic   quantum coarse spaces}:
 
\begin{definition}\label{DefiQuantumCoarseSpace}
Let  $\cM$ be a von Neumann algebra  and  $\mathscr{R}$ be a family of intrinsic quantum relations on $\cM$. The family $\mathscr{R}$ is called an \emph{intrinsic  quantum   coarse structure on $\cM$}   if
\begin{enumerate}
\item\label{Item1Def} $\Delta_{\cM}=\{(p,q): \, pq\neq 0\} \in \mathscr{R}$,
\item\label{Item2Def} $\cR_1\in \mathscr{R}$, $\cR_2\in \mathrm{IQRel}(\cM)$ and $\cR_2\subset  \cR_1$ imply $\cR_2\in \mathscr{R}$, 
\item \label{Item3Def} $\cR_1 \in \mathscr{R}$ implies $\cR_1^{-1}\in \mathscr{R}$,
\item\label{Item4Def} $\cR_1,\cR_2\in \mathscr{R}$ implies $\cR_1\cup\cR_2\in \mathscr{R}$, and
\item\label{Item5Def} $\cR_1,\cR_2\in \mathscr{R}$ implies $\cR_1\circ \cR_2\in \mathscr{R}$.
\end{enumerate}
The pair $(\cM,\mathscr{R})$ is called an \emph{intrinsic   quantum  coarse space}. 
\end{definition}

 Theorem \ref{ThmWeaverQRelAndIntQRel} provides a canonical bijection between $\mathrm{QRel}(\cM \subset  \cB(H))$ and $\mathrm{IQRel}(\cM)$, and the discussion above shows that this bijection preserves the  natural operations on each of them.  We sum this up in the following corollary.
 
 \begin{corollary}\label{CorQRelAndIQRel}
 Let $\cM\subset  \cB(H)$ be a von Neumann algebra. The bijective assignment 
 \[\cV\in \mathrm{QRel}(\cM \subset  \cB(H))\mapsto \cR_\cV\in \mathrm{IQRel}(\cM)\]
 defined in Theorem \ref{ThmWeaverQRelAndIntQRel} satisfies the following:
 \begin{enumerate}
 \item\label{ItemQRelAndIReL1} $\cR_{\cM'}=\Delta_\cM$,
 \item\label{ItemQRelAndIReL2} if $\cV \in  \mathrm{QRel}(\cM\subset  \cB(H))$, then $\cR_{\cV^*}=\cR_{\cV}^{-1}$,
 \item\label{ItemQRelAndIReL3}if $\cV_1,\cV_2\in \mathrm{QRel}(\cM\subset  \cB(H))$, then $\cR_{  \cV_1+\cV_2  }=\cR_{\cV_1}\cup \cR_{\cV_2}$, and 
 \item\label{ItemQRelAndIReL4} if $\cV_1,\cV_2\in \mathrm{QRel}(\cM\subset  \cB(H))$, then $\cR_{ \cV_1\cV_2}=\cR_{\cV_1}\circ \cR_{\cV_2}$.
\end{enumerate}
  \end{corollary}

In light of Corollary \ref{CorQRelAndIQRel}, quantum coarse structures on $\cM \subset  \cB(H)$ and intrinsic quantum coarse structures on $\cM$ naturally correspond, and we denote by $\mathscr{V}\mapsto \mathscr{R}_\mathscr{V}$ and $\mathscr{R}\mapsto\mathscr{V}_\mathscr{R}$ the assignments   implementing the two directions of this correspondence.  We naturally say that an intrinsic quantum coarse space $(\cM,\mathscr{R})$ is  operator reflexive  if $\mathscr{V}_{\mathscr{R}}$ is operator reflexive.

Again, we may abridge notation in the commutative case: if $\cM=L_\infty(X,\mu)$  for a finitely decomposable measure space $(X,\mu)$, we say that $\mathscr{V}$ is a \emph{quantum coarse structure on $X$} (resp. $(X,\mathscr{V})$ is a \emph{quantum coarse space}) if $\mathscr{V}$ is a quantum coarse structure on $L_\infty(X,\mu) \subset  \cB(L_2(X,\mu))$ (resp. $(L_\infty(X,\mu) \subset  \cB(L_2(X,\mu)),\mathscr{V})$ is a quantum coarse space). We treat intrinsic coarse structures/spaces on commutative algebras similarly.

  \section{Quantum uniform Roe algebras}
  \label{SectionQURA}

The axioms for a coarse space $(X,\cE)$ imply that the set of operators in $\cB(\ell_2(X))$ controlled 
by at least one of the relations in $\cE$ is a unital *-algebra, so its norm closure is a unital \cstar-algebra, called the \textit{uniform Roe algebra of $(X,\cE)$} and denoted by $\cstu(X, \cE)$.  Because of Proposition \ref{PropositionWeaverRelation}, this is just the closed union of the members of the corresponding quantum coarse structure on $\ell_\infty(X) \subset  \cB(\ell_2(X))$.  We take it as a general definition.  
 
 \begin{definition}\label{DefiQuantumURoeAlgGEN}
Let $(\cM \subset  \cB(H),\mathscr{V})$ be a  quantum coarse space.  
\begin{enumerate}
    \item We call $\cstu[\cM,\mathscr{V}]=\bigcup_{\cV\in \mathscr{V}} \cV$  the \emph{algebraic quantum uniform Roe algebra of $(\cM,\mathscr{V})$}.
    \item We call $\cstu(\cM,\mathscr{V})=\overline{\bigcup_{\cV\in \mathscr{V}} \cV}^{\|\cdot\|}$ the   \emph{quantum uniform Roe algebra of $(\cM,\mathscr{V})$}.
\end{enumerate}
\end{definition}
  
  We have that $\cstu[\cM,\mathscr{V}]$ is a $^*$-algebra and $\cstu(\cM, \mathscr{V})$ is a $\mathrm C^*$-algebra. Classical uniform Roe algebras are thought of as encoding large scale geometry, and much of the motivation comes from metric spaces.  We see analogous potential in the quantum metric spaces, as introduced in \cite{KuperbergWeaver2012MemAMS} and discussed above.  Definition \ref{DefiQuantumURoeAlgGEN} is general enough, however, to include much more than metric inspiration.  In Section \ref{SubsectionDimFunc} we present an entirely new class of examples.

 \subsection{Basic properties} \label{basicprop}

\subsubsection{The effect of choosing a different representation of $\cM$.} \label{diffrep} Given two faithful representations $\pi_j: \cM \to \cB(H_j)$, $j\in \{1,2\}$, the correspondence between $\mathrm{QRel}(\pi_1(\cM) \subset  \cB(H_1))$ and $\mathrm{QRel}(\pi_2(\cM) \subset  \cB(H_2))$ is described after Theorem \ref{ThmWeaverQRelAndIntQRel}.  Each quantum relation is transformed in the same way as the underlying Hilbert space in passing from one representation to the other by an amplification, spatial isomorphism, and reduction.  As noted in \cite[Theorem 2.7]{Weaver2012MemoirAMS}, this correspondence preserves diagonals (i.e., takes $\pi_1(\cM)'$ to $\pi_2(\cM)'$), inclusions, adjoints, and the operations $(\cV_1, \cV_2) \mapsto \overline{\cV_1 +\cV_2}^{w^*}$ and $(\cV_1, \cV_2) \mapsto\overline{\text{span}}^{w^*}{(\cV_1 \cV_2)}$.  It follows that the possible quantum coarse structures and (algebraic) quantum uniform Roe algebras on $\pi_1(\cM) \subset  \cB(H_1)$ and $\pi_2(\cM) \subset  \cB(H_2)$ also transform under change of representation in the ways given after Theorem \ref{ThmWeaverQRelAndIntQRel}: from $\mathscr{V}$ to $\mathscr{V} \bar{\otimes} \cB(K)$, $u\mathscr{V}u^*$, or $(p'\mathscr{V} p') \mid_{p'H}$.  The only fact that may not be obvious is that amplification by a Hilbert space $K$ commutes with the quantum uniform Roe algebra construction, as the sets involve closures in different orders:
$$\cstu(\cM,\mathscr{V}) \bar{\otimes} \cB(K)  =\overline{\bigcup_{\cV\in \mathscr{V}} \cV}^{\|\cdot\|} \bar{\otimes} \cB(K) =\overline{ \bigcup_{\cV\in \mathscr{V}} \cV \bar{\otimes} \cB(K)}^{\|\cdot\|} = \cstu(\cM,\mathscr{V}\bar{\otimes} \cB(K)).$$
The outside equalities are definitions.  An element of the second set has the form $\sum_{i,j} (a_{ij} \otimes e_{ij})$, where the sum is $w^*$-convergent and each $a_{ij}$ is the norm limit of some $\{v^k_{ij}\}_k \subset \cup_{\cV\in \mathscr{V}} \cV$.  Passing to subsequences if necessary, we may assume $\|v^k_{ij} - a_{ij}\| \leq 2^{-i-j-k}$, and then $ \sum (a_{ij} \otimes e_{ij})$ is the norm limit of $\sum (v^k_{ij} \otimes e_{ij})$, so an element of the third set.  The converse implication amounts to the observation that when a matrix of operators converges in norm, so do each of the entries.

Thus, to understand the quantum coarse structures and quantum uniform Roe algebras that can be associated to $\cM$, it suffices to work with a specific representation of $\cM$.

\subsubsection{Starting only with an intrinsic coarse space.} \label{subsectioniqura} Given an intrinsic quantum coarse space $(\cM, \mathscr{R})$, a representation $\pi: \cM \to \cB(H)$ gives rise to the quantum coarse space $(\pi(\cM) \subset  \cB(H), \mathscr{V}_\mathscr{R})$ and quantum uniform Roe algebra $\cstu(\pi(\cM), \mathscr{V}_\mathscr{R})$.  Since many representations are feasible, $(\cM, \mathscr{R})$ alone only determines an equivalence class
\[ [\pi(\cM) \quad \subset  \quad \cB(H) \quad \supset \quad \cstu(\pi(\cM), \mathscr{V}_\mathscr{R})].\]
Here two such double inclusions are equivalent if one can be obtained from the other by a change of representation of $\cM$ as in Section \ref{diffrep}.  Inspecting the proof of the classical structure theorem for *-isomorphisms of represented von Neumann algebras (\cite[Theorem IV.5.5 and Corollary IV.5.6]{Takesaki2002BookI}), this means that the double inclusions become spatially isomorphic after both are amplified by a single Hilbert space whose dimension is infinite and not less than the dimension of either of the two original Hilbert spaces.  
Theorem \ref{ThmEmbAndIso}(2) treats such ``algebraically/intrinsically identical quantum coarse spaces" from a slightly different perspective.

 \subsubsection{The minimal quantum coarse structure on $\cM \subset  \cB(H)$.}  \label{SecConnected} Since we require quantum coarse structures on $\cM \subset  \cB(H)$ to contain the quantum relation $\cM'$, and $w^*$-closed sub-$\cM'$-bimodules of $\cM'$ are precisely the direct summands, we have the following minimality result.

\begin{proposition}\label{PropQURAcommutant}
For $\cM \subset  \cB(H)$, the collection of direct summands of $\cM'$ is the minimal coarse structure, contained in every other one.  Thus $\cM'$ is the minimal (algebraic) quantum uniform Roe algebra for $\cM \subset  \cB(H)$, contained in every other one.
\end{proposition}

\subsubsection{Connectedness} A (classical) coarse space $(X,\cE)$ is \emph{connected} if $\{(x,y)\}\in \cE$ for all $x, y\in X$.  In the quantum world, this can be described as follows:

\begin{definition}
A quantum coarse structure $\mathscr{V}$ on $\cM \subset  \cB(H)$ is \textit{connected} if $\cstu(\cM,\mathscr{V})$ is $w^*$-dense in $\cB(H)$.  This is equivalent to triviality of $\cstu(\cM,\mathscr{V})'$ (which is always a von Neumann subalgebra of $\cM$, since $\cM' \subset  \cstu(\cM,\mathscr{V})$).
\end{definition}

Notice that $(X,\cE)$ is connected in the classical sense if and only if  $\cstu(X,\mathscr{V}_{\cE})$ is $w^*$-dense in $\cB(\ell_2(X))$, so the definition above coincides with the classical one for $\cM=\ell_\infty(X)\subset  \cB(\ell_2(X))$.    For (classical) coarse structures coming from metric spaces, connectedness means that all distances are finite.  There are analogues of ``all distances finite" for the quantum coarse structures of quantum metric spaces, one of which says that the union of all the intrinsic quantum relations is exactly the \emph{linkable pairs of projections} in $\cM \bar{\otimes} \cB(\ell_2)$ (see \cite[2.11-13]{KuperbergWeaver2012MemAMS} or Definition \ref{Vdistance} below).

Suppose $\mathscr{V}$ is disconnected, and let $p \in \cstu(\cM,\mathscr{V})'$ be a nontrivial projection.  Then every $\cV \in \mathscr{V}$ can be decomposed as $p\cV \oplus (1-p)\cV$, and each $p\cV$ is a quantum relation over $p\cM p \subset  \cB(pH)$.  Writing $p\mathscr{V}$ for the collection of $p\cV$, in a straightforward way we have
$$\mathscr{V} = p\mathscr{V} \oplus (1-p)\mathscr{V},$$
where the summands are quantum coarse structures on $p\cM p \subset  \cB(pH)$ and $(1-p)\cM (1-p) \subset  \cB((1-p)H)$, respectively. Thus connectedness is the nonexistence of a direct sum decomposition for $\mathscr{V}$.

For a classical coarse structure $\mathscr{V}$ (i.e., $\mathscr{V}=\mathscr{V}_\cE$ for some coarse structure $\cE$ on a set $X$), the commutant $\cstu(\ell_\infty,\mathscr{V})'$ is a von Neumann subalgebra of $\ell_\infty(X)$, so it has minimal projections.  The minimal projections correspond to connected classical coarse structures occurring as summands of $\mathscr{V}$, naturally called the \textit{connected components} of $\mathscr{V}$.  If $\mathscr{V}$ arises from a metric space, this is the decomposition into components where the metric is finite; if the metric arises from a graph as in Example \ref{quantumgraph}, this is the decomposition into connected components of the graph.

A classical coarse structure with finitely many connected components can be recovered from them.  If there are infinitely many components, this recovery is not generally possible: the collection $\{p\mathscr{V}\}_p$, where $p$ runs over the minimal projections $p \in \cstu(\cM,\mathscr{V})'$, does not generally determine $\mathscr{V}$.  For general quantum coarse structures the term ``connected component" should probably be avoided, as minimal projections in $\cstu(\cM,\mathscr{V})'$ may not commute or even exist at all. 

\subsubsection{Triviality} We now discuss   natural notions of ``triviality'' for quantum coarse structures. Here is a pentachotomy (list of five disjoint cases) for a quantum coarse structure $\mathscr{V}$ on $\cM \subset  \cB(H)$.

\begin{enumerate}
    \item[(1)] Some $\cV \in \mathscr{V}$ is $\cB(H)$; equivalently, $\mathscr{V} = \mathrm{QRel}(\cM \subset  \cB(H))$.  In this case, $\mathscr{V}$ is  metrizable (generated by $\cV=\cB(H)$).
    \item[(2)] No $\cV \in \mathscr{V}$ is $\cB(H)$, but the algebraic quantum uniform Roe algebra is $\cB(H)$.  This case is nonmetrizable and nonclassical.
    \item[(3)] The algebraic quantum uniform Roe algebra is not $\cB(H)$, but its norm closure (the quantum uniform Roe algebra) is.  This can happen classically.  
    \item[(4)] The quantum uniform Roe algebra is not $\cB(H)$, but its $w^*$-closure is.  
    \item[(5)] The quantum uniform Roe algebra is not $w^*$-dense in $\cB(H)$, i.e., the quantum coarse structure is disconnected.
\end{enumerate}
Next we justify the nontrivial claims above and give examples to show that all five cases occur.

Quantum coarse structures in (2) are necessarily nonmetrizable: by the Baire category theorem, if $H$ is infinite dimensional, then $\cB(H)$ is not a countable union of proper closed subspaces.  An example is the collection of finite-dimensional subspaces of $\cB(H)$, which is a quantum coarse structure for $\cB(H) \subset  \cB(H)$.  This phenomenon cannot happen for classical algebraic uniform Roe algebras:   if an operator $a=(a_{xy})_{x,y}$ on $\ell_2(X)$ has only nonzero coordinates, i.e., $a_{xy}\neq 0$ for all $x,y\in X$, and $a$   belongs to the  algebraic uniform Roe algebra of a coarse space $(X,\cE)$, then $X\times X\in \cE$; so $B(\ell_2(X))\in \cV_{\cE}$ and we are in case (1).

\begin{example}\label{ExampleNonRigidityLF}
We give a classical coarse structure in case (3). Recall that a relation $E$ on $\mathbb{N}$ is said to be \textit{locally finite} if   \[\{m : \, (m,n) \in E\}\ \text{ and }\ \{m : \, (n,m) \in E\}\] are finite for each $n\in\N$.  The collection $\mathcal{E}$ of all locally finite relations is easily seen to be a classical coarse structure.  We have $\cstu[\cM,\mathcal{E}] \neq \cB(\ell_2)$, because an operator with all nonzero entries is not controlled by a locally finite relation.

On the other hand, take any $\varepsilon > 0$ and $a \in \cB(\ell_2)$.  The $k$th column of $a$ has finite $\ell_2$ norm, so the tail beyond some point has $\ell_2$ norm $< \varepsilon/2^{k+1}$; change all these entries to zero.  Then do the same thing for the rows of $a$.  This produces an operator $b \in \cstu[\cM,\mathcal{E}]$ with the Hilbert-Schmidt norm of $a-b$ (which is the $\ell_2$ norm of its entries) less than $\varepsilon$.  The Hilbert-Schmidt norm dominates the operator norm, so $\cstu(\cM,\mathcal{E}) = \cB(\ell_2)$.
\end{example}

We make no ruling on the meaning of ``large-scale triviality".  There are at least three choices:
\begin{itemize}
    \item some $\cV \in \mathscr{V}$ is $\cB(H)$ --- case (1) in the pentachotomy;
    \item the algebraic quantum uniform Roe algebra is $\cB(H)$ --- cases (1) and (2);
    \item the quantum uniform Roe algebra is $\cB(H)$ --- cases (1), (2), and (3).
\end{itemize}
Depending on how much detail the reader's quantum telescope allows, he or she may decide which of these quantizes the coarse geometer's slogan: the structure ``looks like a quantum point from far away."


\subsection{Examples} \label{examples}

 \subsubsection{$\cM \simeq \ell_\infty(X)$.}  As mentioned at the beginning of this section, quantum uniform Roe algebras for $\ell_\infty(X) \subset  \cB(\ell_2(X))$ are nothing but classical uniform Roe algebras for $X$.

  \subsubsection{$\cM \simeq \mathrm{M}_n(\C)$.} \label{quraMn}
 Considering $\cM$ acting on $\C^n$, i.e., $\cM=\mathrm{M}_n(\C)=\cB(\C^n)$, we have that $\cM'=\C1_{\C^n}$. By finite dimensionality, the  quantum relations of $\cM\subset  \cB(\C^n)$   are just the linear subspaces of $\cB(\C^n) $. Therefore, since quantum coarse structures contain $\cM'$ and are closed under subspaces, adjoints, sums, and products of quantum relations, every quantum  coarse structure on $\mathrm{M}_n(\C) \subset  \cB(\C^n)$ is the collection of subspaces of some  unital $^*$-subalgebra   $\cA\subset  \mathrm{M}_n(\C) $.  For such a quantum coarse structure, the quantum uniform Roe algebra is $\cA$.
 
 \subsubsection{$\cM \simeq \cB(H)$ (infinite-dimensional)} \label{B(H)qura}
 Some of the analysis above applies: with $\cM$ acting on $H$, quantum relations are $w^*$-closed subspaces of $\cB(H)$.  Any unital *-subalgebra $\cA_0 \subset  \cB(H)$ is the algebraic quantum uniform Roe algebra for the quantum coarse structure consisting of the finite-dimensional subspaces of $\cA_0$.  Thus, any unital $\mathrm C^*$-subalgebra $\cA \subset  \cB(H)$ can be obtained as a quantum uniform Roe algebra for $\cM=\cB(H) \subset  \cB(H)$ by choosing the quantum coarse structure consisting of finite-dimensional subspaces of a norm-dense *-algebra $\cA_0 \subset  \cA$.
 
 There may be other quantum coarse structures generating $\cA$ (or $\cA_0$).  We point out, though, that including certain combinations of $w^*$-closed subspaces of $\cA$ may make the generated coarse structure too large, in the sense that it controls operators outside $\cA$.  Here is an example.  Let $\cA$ be the unital *-algebra $\ell_\infty(\{\text{odds}\}) + c(\{\text{evens}\})$, thought of as diagonal operators on $\ell_2(\mathbb{N})$ --- here $c$ denotes convergent sequences (the unitization of $c_0$).  Let $\cV_1=\ell_\infty(\{\text{odds}\})$ and $\cV_2=\overline{\text{span}}^{w^*}\{e_{2j-1,2j-1} + (e_{2j,2j}/j)\}_{j=1}^\infty$.  Then $\cV_1$ and $\cV_2$ are $w^*$-closed subspaces of $\cA$, but $\overline{\cV_1+\cV_2}^{w^*}$ is $\ell_\infty(\{\text{odds}\}) + \ell_\infty(\{\text{evens}\})$, the full diagonal.

 \subsubsection{Finite-dimensional $\cM$}
 
 \begin{proposition}\label{contcomm}
 Let $\cA, \cM \subset  \cB(H)$, where $\cA$ is a *-algebra and $\cM$ is a finite-dimensional von Neumann algebra, both containing $1_H$.  Then $\cA$ is an algebraic quantum uniform Roe algebra for some quantum coarse structure on $\cM \subset  \cB(H)$ if and only if  $\cA$ contains $\cM'$.  When this is the case, $\cA$ is automatically norm-closed, i.e., a $\mathrm C^*$-algebra.
 \end{proposition}
 
 \begin{proof}
 Since the quantum uniform Roe algebra of  a quantum coarse space $(\cM,\mathscr{V})$  always contains $\cM'$,  only the reverse implication requires proof. As $\cM$ is finite dimensional, suppose $\cM \simeq \bigoplus_{k=1}^\ell \mathrm{M}_{n_k}(\C)$, and let $n = \sum_{k=1}^\ell n_k$.  Denote by $z_1, \ldots, z_\ell$ the minimal central projections of $\cM$. Since we assume that $\cM'\subset  \cA$, we have that $z_1,\ldots, z_\ell\in \cA$. 
 
 First consider a multiplicity-free representation of $\cM$ on $\C^n$, i.e., a minimal faithful representation, with commutant equal to the center of $\cM$; that is, $\cM=\bigoplus_{k=1}^\ell \mathrm{M}_{n_k}(\C)\subset \cB(\C^n)$.    Let $\mathscr{V}$ be the set of all   subspaces of $\cB(\C^n)$ of the form  $\sum_{i,j=1}^\ell V_{ij}$, where each $V_{ij}$ is a subspace of $ z_i \cA z_j \subset  \cA$.  These are $\cM'$-bimodules whose union is $\cA$, and it is straightforward to check that they satisfy the axioms for a quantum coarse structure (weak$^*$-closedness follows automatically from finite dimensionality). So $\cstu(\cM,\mathscr{V})=\cstu[\cM,\mathscr{V}]=\cA$. 
 
 An arbitrary representation $\cM\subset  \cB(H)$ can be obtained from the multiplicity-free representation above  by amplification, spatial isomorphism, and reduction.  Each of these operations preserves the relation that $\cM' \subset  \cA$, because all quantum relations, including the diagonal quantum relation $\cM'$ and all the other quantum relations whose union is $\cA$, are transformed in the same way.  To recall from the discussion after Theorem \ref{ThmWeaverQRelAndIntQRel}: all are tensored with some $\cB(K)$, all are conjugated by a unitary with domain $H$, or all are reduced by a projection $p' \in \cM'$. Since these transformations also preserve  closedness, this finishes the proof.
  \end{proof}
  
Proposition \ref{contcomm} characterizes the quantum uniform Roe algebras for a finite-dimensional $\cM \subset  \cB(H)$ as those $\mathrm C^*$-subalgebras of $\cB(H)$ that contain $\cM'$.  This characterization is also true for $\cM \simeq \cB(H)$, as seen above in Section \ref{B(H)qura}.  The reader may wonder if it is generally true.  Our next example shows that this fails already in the classical case.

\subsubsection{A \cstar-algebra $\cA \subset  \cB(\ell_2)$ that contains the diagonal $\ell_\infty$ but is not a quantum uniform Roe algebra for $\ell_\infty \subset  \cB(\ell_2)$} Notice that, a  fortiori, $\cA$ is   not an algebraic quantum uniform Roe algebra for $\ell_\infty \subset  \cB(\ell_2)$ either. Since this is the same as saying that $\cA$ is not a uniform Roe algebra in the classical sense,   we use non-quantum terminology below.

An operator $x=(x_{ij})_{i,j} \in \cB(\ell_2)$ is called a \textit{ghost} if its entries $x_{ij}$ go to zero as $i,j \to \infty$.  Letting $p_n$ be the projection on $\C^n$ with all entries $1/n$ (i.e., $p_n$ is  the projection onto the constant vector $(1,\ldots, 1)\in \C^n$), we have that $p=\oplus_n p_n \in \cB(\ell_2)$ is a noncompact ghost projection.  Let $E_\text{diag}$ be the conditional expectation from $\cB(\ell_2)$ onto its diagonal $\ell_\infty$, i.e., $E_{\text{diag}}((  x_{ij})_{i,j})=(\bar x_{ij})_{i,j}$ for all $(  x_{ij})_{i,j}\in \cB(\ell_2)$, where $\bar x_{ii}=x_{ii}$ for all $i\in\N$,  and $\bar x_{ij}=0$ for all $i\neq j$. It is straightforward to check that    $x - E_{\text{diag}}(x)$ is a ghost for all $x$ in the \cstar-algebra $\cA = \mathrm C^*(p, \ell_\infty)$.

Suppose towards a contradiction that $\cA$ is a uniform Roe algebra, say $\cA=\cstu(\N,\cE)$ for some coarse structure $\cE$ on $\N$.  Then there is an entourage $E\in\cE$ and $q\in \cB(\ell_2)$ controlled by $E$ with  $\|q-p\| < 0.1$.  Notice that in each of the blocks corresponding to the summands of $p$ (except the $1 \times 1$ block), $q$ must have a nonzero offdiagonal entry.  (Otherwise, compare the actions of $p$ and $q$ on the unit vector with all entries corresponding to that block equal to ${1}/{\sqrt{n}}$.)

Choosing a nonzero offdiagonal entry for $q$ in each block, let $x$ be the operator whose entries are 1 in these places and 0 elsewhere.  Then $x$ is also controlled by $E$, so $x \in \cA$.  But $x - E_\text{diag}(x) = x$ is not a ghost.

\section{A new class of examples: support expansion C*-algebras}\label{SubsectionDimFunc}

The main inspiration for coarse geometry comes from metric spaces.  The standard example of a coarse space is a metric space $(X,d)$, where the entourages are subsets of  some  $\{(x,y) : \,  \: d(x,y) \leq r\}$ for   $r
\geq 0$.  In the associated uniform Roe algebra, the operators supported on an entourage are those for which there is an $r\geq 0$ such that ``no point moves more than $r$."  These operators could be said to have a finite scale in the sense of \textit{displacement}.

The definitions of coarse geometry, classical or quantum, are broad enough to allow other notions of scale.  In this section we present a very general class of quantum coarse spaces in which the scale pertains not to change of location as determined by a metric, but to \textit{expansion of size} as determined by a \textit{measure}.  The reader may see \cite{EisnerThesis} for a detailed study of the commutative (but not necessarily atomic) case.

Let us describe the prototype for the new class, which is a well-known classical coarse space with classical uniform Roe algebra.  A coarse structure $\cE$ on $\N$ is said to be \textit{uniformly locally finite} if for any one of its entourages $E\in \cE$ there is $\lambda>0$ such that for any $x \in \N$ both $\#\{y\in X : \, (x,y) \in E\}$ and $\#\{y\in X : \, (y,x) \in E\}$ are less than $\lambda$.  In a metric setting  this says that for any $r$, the cardinalities of $r$-balls are uniformly bounded (so-called ``bounded geometry").  It is easy to see that the collection of all entourages satisfying the above condition comprises the largest uniformly locally finite coarse structure on $\N$.  This well-known example is not metrizable, and the associated algebraic uniform Roe algebra consists of all operators $a \in \cB(\ell_2)$ for which there is a $\lambda$ such that all rows and columns of the matrix for $a$ have no more than $\lambda$ nonzero elements.  In other words, if $(e_j)_j$ is the canonical orthonormal basis of $\ell_2$, the   supports of  $ae_j$ and $ a^*e_j$ have no more than $\lambda$ elements for all $j\in\N$.  Taking linear combinations of basis elements, it follows that $a$ satisfies the condition
\begin{equation} \label{suppscale}
\# \text{supp}(a \xi), \: \text{supp}(a^* \xi) \leq \lambda \cdot \# \text{supp}(\xi), \qquad \forall \xi \in \ell_2.
\end{equation}
The associated uniform Roe algebra, studied in \cite{Manuilov2019}, is  the closure of these operators.

We may identify subsets of $\N$ with their characteristic functions, which are the projections in $\ell_\infty$.  The support of $\xi \in \ell_2$, which is a set, corresponds to $s^{\ell_\infty}(\xi)$.  Counting the support of $\xi$ is then the same thing as applying the standard $\ell_\infty$ trace to $s^{\ell_\infty}(\xi)$.  We can interpret \eqref{suppscale} as saying that $a$ and $a^*$ do not expand the support projection of a vector too much.

Many things about \eqref{suppscale} allow for variation.  For instance, instead of the counting measure, we could use other measures on $\N$, corresponding to other (possibly unbounded) traces on $\ell_\infty$. Similarly, we could use $L_2(X, \mu)$ and $L_\infty(X,\mu)$ in place of $\ell_2$ and $\ell_\infty$.  We could even  use an abstract $H$ with an arbitrary von Neumann subalgebra $\cM \subset  \cB(H)$, equipped with a satisfactory notion of ``size" on its projection lattice.  Precisely:

\begin{definition}
Let $\cM \subset  \cB(H)$ be a von Neumann algebra equipped with a faithful normal semifinite trace $\tau$. Given $\lambda\geq 0$, we   say that $a\in \cB(H)$ is \textit{$\lambda$-vector constrained} if it satisfies
\begin{equation} \label{suppscale2}
\tau(s^\cM(a \xi)), \: \tau(s^\cM(a^* \xi)) \leq \lambda \cdot \tau(s^\cM(\xi)), \qquad \forall \xi \in H.
\end{equation}
We say $a$ is  \textit{vector constrained} if it is $\lambda$-constrained for some $\lambda\geq 0$.
\end{definition}

\begin{proposition} \label{vecconstrain}
Let $\cM \subset  \cB(H)$  be a von Neumann algebra equipped with a faithful normal semifinite trace $\tau$. The subset of all vector constrained operators forms a unital *-algebra.
\end{proposition}

\begin{proof}
It is immediate that this set contains the identity and is closed under adjoint.  We need to show that if $a_j$ is $\lambda_j$-constrained for $j\in\{1,2\}$, then $a_1 + a_2$ and $a_1 a_2$ are constrained.  These are easy computations:

$$\tau(s^\cM((a_1 + a_2) \xi)) \leq \tau(s^\cM(a_1 \xi) \vee s^\cM(a_2 \xi)) \overset{\heartsuit}{\leq} \tau(s^\cM(a_1 \xi)) + \tau(s^\cM(a_2 \xi)) \leq (\lambda_1 + \lambda_2) \tau(s^\cM(\xi));$$

$$\tau(s^\cM(a_1 a_2 \xi)) \leq \lambda_1 \cdot \tau(s^\cM(a_2 \xi)) \leq \lambda_1 \lambda_2 \tau(s^\cM(\xi)).$$
We applied \eqref{E:suppproperties} at the first step of the first computation.  The inequality $\heartsuit$ is a standard deduction of Kaplansky's law $p \vee q - p \sim q - p\wedge q$, so that $\tau(p \vee q) = \tau(p) + \tau(q) - \tau(p \wedge q) \leq \tau(p) + \tau(q)$.  We use this freely in the sequel.
\end{proof}

\begin{definition}\label{DefVectorSuppExp}
Let $\cM \subset  \cB(H)$  be a von Neumann algebra equipped with a faithful normal semifinite trace $\tau$. The closure of all vector constrained operators on $\cB(H)$ is called   a \emph{vector support expansion \cstar-algebra}. 
\end{definition}

Motivated by $\cM=\ell_\infty$ as discussed at the beginning of this section, one might expect that vector support expansion \cstar-algebras are quantum uniform Roe algebras.  In many cases of interest this is indeed the case, but the general answer turns out to be no!  The issue, roughly, is that quantum uniform Roe algebras are about projections, and it is not always true that projections in a represented von Neumann algebra ``come from" vectors.  We explain this more in Section \ref{SubsectionAbelianMeasDistorsion}.

For now we recast expansion in terms of projections, and we show that this does lead to a rich class of quantum uniform Roe algebras.

\subsection{Support expansion C*-algebras} \label{SubsectionExamDimFun}

Let $\cM \subset  \cB(H)$, and let $\tau$ be a faithful normal semifinite trace on $\cM$.  For $\lambda \geq 0$, consider the following condition on $a \in \cB(H)$:
\begin{equation} \label{suppscale3}
\tau(s_\ell^\cM(a q)), \: \tau(s_\ell^\cM(a^*q)) \leq \lambda \cdot \tau(q), \qquad \forall q \in \Proj(\cM).
\end{equation}
The condition above is a strengthening   of \eqref{suppscale2} (see the first paragraph of the proof of Theorem \ref{abelianexpansion}). One can use the same argument as in Proposition \ref{vecconstrain} to see that the set of operators in $\cB(H)$ satisfying \eqref{suppscale3} for some $\lambda$ is a unital *-algebra; this is subsumed in Theorem \ref{thmprojsuppexp} below, where we identify this set as an algebraic quantum uniform Roe algebra.

\begin{definition} \label{constrainedphi}
Given $\lambda\geq 0$,   a function $\varphi: \Proj(\cM) \to \Proj(\cM)$ is \textit{$\lambda$-constrained} if \[\tau(\varphi(q)) \leq \lambda \tau(q)\ \text{  for all }\ q \in \Proj(\cM).\]  We say $\varphi$ is \textit{constrained} if it is $\lambda$-constrained for some $\lambda$. 
\end{definition}

Note that the set of constrained functions is closed under composition and join: if $\varphi_j$ is $\lambda_j$-constrained for $j\in \{1,2\}$,
$$\tau(\varphi_2(\varphi_1(q))) \leq \lambda_2 \tau(\varphi_1(q)) \leq \lambda_2 \lambda_1 \tau(q);$$
$$\tau(\varphi_1(q) \vee \varphi_2(q)) \leq \tau(\varphi_1(q)) + \tau(\varphi_2(q)) \leq (\lambda_1 + \lambda_2) \tau(q). $$
For any pair of constrained functions $\varphi, \psi: \Proj(\cM) \to \Proj(\cM)$, we define \[\cV_{\varphi, \psi} = \cV_\varphi \cap (\cV_\psi)^*,\] where, as defined in Example \ref{oprefex},  
\[\cV_\varphi=\{a\in \cB(H) : \, s^\cM_\ell(aq)\leq \varphi(q),\ \forall q\in \mathrm{Pr}(\cM)\}.\]
As noted in  Example \ref{oprefex}, $\cV_{\varphi, \psi}$ is an operator reflexive quantum relation on $\cM\subset  \cB(H)$.  The main statement of the next theorem is that the quantum coarse structure generated by the $\cV_{\varphi, \psi}$ is simply the collection of their quantum subrelations.

\begin{theorem} \label{thmprojsuppexp}
Let $\cM \subset  \cB(H)$ and let $\tau$ be a faithful normal semifinite trace on $\cM$.  The collection
\[\mathscr{V}_\tau = \{\cV\in \mathrm{QRel}(\cM\subset  \cB(H)): \, \exists \text{ constrained }\varphi, \psi:\mathrm{Pr}(\cM)\to\mathrm{Pr}(\cM)\text{ such that }\cV\subset  \cV_{\varphi,\psi}\}\]
is an operator reflexive quantum coarse structure on $\cM$.  Moreover, the   algebraic quantum uniform Roe algebra $\cstu[\cM,\mathscr{V}_\tau]$ is exactly the set of $a \in \cB(H)$ satisfying \eqref{suppscale3} for some $\lambda\geq 0$.  
\end{theorem}

\begin{proof}
Let us check that $\mathscr{V}_\tau$ has the five properties  needed to be a quantum coarse structure (see Definition \ref{DefiQuantumCoarseSpaceOriginal}).  From its form, $\mathscr{V}_\tau$ is closed under sub-quantum relation and adjoint.

\textit{Diagonal.}  The identity  $\Proj(\cM)\to\Proj(\cM) $ is a 1-constrained function on $\Proj(\cM)$.  Observe that $\cM' \subset  \cV_{\text{id}, \text{id}}$ because for $a \in \cM'$ and $q\in \mathrm{Pr}(\cM)$ we have $s_\ell^\cM(aq) = s_\ell^\cM(qaq) \leq q$.

\textit{Sum.} Suppose $a_j \in \cV_{\varphi_j, \psi_j}$ for $j\in \{1,2\}$.  Then $a_1 + a_2 \in \cV_{\varphi_1 \vee \varphi_2, \psi_1 \vee \psi_2}$: for any $q \in \Proj(\cM)$,
$$s_\ell^\cM((a_1 + a_2)q) \leq s_\ell^\cM(a_1 q) \vee s_\ell^\cM(a_2 q) \leq \varphi_1(q) \vee \varphi_2(q),$$
and a similar computation holds for $s_\ell^\cM((a_1 + a_2)^*q)$.  We noted before the theorem that $\varphi_1 \vee \varphi_2$ and $\psi_1 \vee \psi_2$ are constrained.

\textit{Product.} Again suppose $a_j \in \cV_{\varphi_j, \psi_j}$ for $j\in \{1,2\}$.  Then $a_1 a_2 \in \cV_{\varphi_1 \circ \varphi_2, \psi_2 \circ \psi_1}$: for any $q \in \Proj(\cM)$,
$$\varphi_1(\varphi_2(q)) a_1 a_2 q = \varphi_1(\varphi_2(q)) a_1 \varphi_2(q) a_2 q = a_1 \varphi_2(q) a_2 q = a_1 a_2 q,$$
and a similar computation holds for $\psi_2(\psi_1(q)) (a_1 a_2)^* q$.  We noted before the theorem that $\varphi_1 \circ \varphi_2$ and $\psi_2 \circ \psi_1$ are constrained.

Since the $\cV_{\varphi, \psi}$ are operator reflexive, $\mathscr{V}_\tau$ is an operator reflexive quantum coarse structure.

Finally we check that the algebraic quantum uniform Roe algebra, the union of the $\cV_{\varphi, \psi}$, is exactly the set of operators $T$ satisfying \eqref{suppscale3} for some $\lambda$.  For this, it suffices to show that $\tau(s_\ell^\cM(aq)) \leq \lambda \tau(q)$ for all $q \in \Proj(\cM)$ if and only if $a$ belongs to some $\cV_\varphi$ with $\varphi$ being $\lambda$-constrained.  The forward implication follows by letting $\varphi(q) = s_\ell^\cM(aq)$.  For the reverse implication, take any $q \in \Proj(\cM)$ and note that $\tau(s_\ell^\cM(aq)) \leq \tau(\varphi(q)) \leq \lambda \tau(q)$.
\end{proof}

The following definition should be compared with Definition  \ref{DefVectorSuppExp} (the ``vector" version).

\begin{definition}\label{DefProjSuppExp}
Let $\cM \subset  \cB(H)$, and let $\tau$ be a faithful normal semifinite trace on $\cM$. We say that the uniform Roe algebra   $\cstu(\cM,\mathscr{V}_\tau)$ is a \emph{support expansion \cstar-algebra}.
 \end{definition}

From the definition of a quantum coarse structure, $\cM'$ lies inside any algebraic quantum uniform Roe algebra.  We next show that support expansion \cstar-algebras also contain $\cM$ yet cannot be all of $\cB(H)$ unless $\cM$ is finite-dimensional.

\begin{theorem}\label{PropDimContMinside}
Let $\cM \subset  \cB(H)$, $\tau$ be a faithful normal semifinite trace on $\cM$, and $\mathscr{V}_\tau$ be the quantum coarse structure as in Theorem \ref{thmprojsuppexp}.
\begin{enumerate}
\item The algebraic uniform Roe algebra $\cstu[\cM,\mathscr{V}_\tau]$ contains $\cM$, and thus   $\cstu(\cM,\mathscr{V}_\tau)$ contains $\mathrm{C}^*(\cM \cup \cM')$.

\item If $\cM$ is infinite-dimensional,   $\cstu(\cM,\mathscr{V}_\tau)$ is not all of $\cB(H)$.
\end{enumerate}
\end{theorem}

\begin{proof}
(1) We make  use of the \textit{right $\cM$-support $s_r^\cM$} of an operator, which is defined analogously to $s^\cM_\ell$: $s_r^\cM(a)$ is the smallest projection $q\in \cM$ with $aq=a$.  Note that the left and right $\cM$-supports of any element of $\cM$ are Murray-von Neumann equivalent via the partial isometry in the polar decomposition.  Thus any $a \in \cM$ is 1-constrained:
$$\tau(s^\cM_\ell(aq)) = \tau(s^\cM_r(aq)) \leq 1 \cdot \tau(q), \qquad \forall q \in \Proj(\cM),$$
and analogously for $a^*$.

(2) It follows from Section \ref{diffrep} that we may assume $\cM$ is in ``standard form," which we recall for the reader's convenience.  On the subspace of $x \in \cM$ for which $\|x\|_2 = \sqrt{\tau(x^*x)}$ is finite, this quantity defines a norm.  The completion of this normed space is denoted $L^2(\cM, \tau)$, and $\cM \subset  \cB(L_2(\cM, \tau))$ as densely-defined left multiplication operators.

First we explain how to pick sequences  $(p_k)_{k\in\N}$ and $(q_k)_{k\in\N}$   of nonzero pairwise orthogonal finite-trace projections in $\cM$ with $\lim_k\frac{\tau(p_k)}{\tau(q_k)}=0$. If $\tau$ is a finite trace, then $\lim_k\tau(q_k)= 0$ for any sequence  $(q_k)_{k\in\N}$ of nonzero pairwise orthogonal projections in $\cM$, and we may take $(p_k)_{k\in\N}$ to be   a subsequence of any such $(q_k)_{k\in\N}$ whose   traces decrease quickly enough.  If $\tau$ is infinite,  semifiniteness of $\tau$ allows us to  find an infinite family of pairwise orthogonal projections $(p_k)_{k\in\N}$ each of which has  finite trace at least 1. We can then  form the sequence $(q_k)_{k\in\N}$ by letting each $q_k$ be an appropriate   sum of disjoint subsets of $(p_k)_{k\in\N}$.

Fix $(p_k)_{k\in\N}$ and $(q_k)_{k\in\N}$  as above. If $p\in \cM$, we let $\hat p$ be $p$ viewed as an element in $L_2(\cM, \tau)$.  Let $b \in \cB(L_2(\cM,\tau))$ be the infinite-rank partial isometry that sends each $\hat{p}_k$ to $\sqrt{\frac{\tau(p_k)}{\tau(q_k)}} \hat{q}_k$ and is zero off the span of $\{\hat{p}_k\}$. We show that for any $a$ satisfying \eqref{suppscale3}, $\|b-a\| \geq 1$ by examining the action on unit vectors $\frac{\hat{p_k}}{\sqrt{\tau(p_k)}}$.  This implies that $b\not\in \cstu(\cM,\mathscr{V}_\tau)$.  

Let $s_k=s_\ell^{\cM}(a \cdot p_k)$, where $a \cdot p_k \in \cB(L_2(\cM, \tau))$ means left multiplication by $p_k$ followed by $a$.  We have $(1-s_k)a(\hat p_k) = (1-s_k) (a \cdot p_k) (\hat p_k) = 0$ and, from the assumption on $a$, $\tau(s_k)\leq \lambda \tau(p_k)$.  In the calculation below, $(*)$ is left multiplication by $(1-s_k)$ and $(**)$ (read right-to-left) is right multiplication by $q_k$, both contractive:

\begin{align*}
\Big\|(b-a)\Big(\frac{\hat p_k}{\sqrt{\tau(p_k)}}\Big)\Big\|_2 &=    \Big\|\frac{\hat q_k}{\sqrt{\tau(q_k)}}- a\Big(\frac{\hat p_k}{\sqrt{\tau(p_k)}}\Big)\Big\|_2
\\ &\overset{(*)}{\geq} \Big\|(1-s_k)\Big(\frac{\hat q_k}{\sqrt{\tau(q_k)}} \Big)\Big\|_2
\\ &\geq \Big\| \frac{\hat q_k}{\sqrt{\tau(q_k)}}\Big\|_2 -\Big\|  \frac{\widehat{s_k q_k}}{\sqrt{\tau(q_k)}}  \Big\|_2
\\ &\overset{(**)}{\geq} 1-\Big\|  \frac{\hat{s_k}}{\sqrt{\tau(q_k)}}  \Big\|_2 = 1 - \sqrt{\frac{\tau(s_k)}{\tau(q_k)}} \geq 1-\sqrt{\frac{ \lambda \tau(p_k)}{\tau(q_k)}} \to 1. \qedhere
\end{align*}
\end{proof}

\begin{remark}\label{RemarkSquareEx}

In the construction of support expansion \cstar-algebras just described, the expansion we tolerate in operators $a$ and projection maps $\varphi$ is constrained by multiplication by some constant, that is, a linear function.  The definitions also make sense for other functions.  Given an increasing $f:[0,\infty] \to [0,\infty]$, we may widen Definition \ref{constrainedphi} and say that a projection map $\varphi: \Proj(\cM) \to \Proj(\cM)$ is \textit{$f$-constrained} if $\tau(\varphi(q)) \leq f(\tau(q))$  for all $q \in \Proj(\cM)$.

To obtain the generalization of Theorem \ref{thmprojsuppexp}, we need to start with a nonempty collection of functions $\mathscr{F}$ that is closed under sums and compositions.  This entails that the operators belonging to some $\cV_{\varphi, \psi}$, where $\varphi$ and $\psi$ are each constrained by some member of $\mathscr{F}$, will comprise a *-algebra.

In the present paper the underlying $\mathscr{F}$ is always the collection of functions $\{f_\lambda(x) = \lambda x\}_{\lambda \geq 0}$, but other choices do in fact produce distinct quantum uniform Roe algebras.  One may, for instance, let $\mathscr{F}$ be the collection of functions obtained from $f(x) = \sqrt{x}$ under arbitrary repeated sums and compositions.  The companion paper \cite{EisnerThesis} studies many aspects of this construction for abelian $\cM$ and in particular analyzes the wild poset of different quantum uniform Roe algebras arising from various $\mathscr{F}$ when $\cM$ is $L_\infty(\mathbb{R})$ endowed with the Lebesgue integral as trace.
\end{remark}

\subsection{Vector support expansion C*-algebras}\label{SubsectionAbelianMeasDistorsion}

Let us recall the two conditions in Definitions \ref{DefVectorSuppExp} and \ref{DefProjSuppExp} on operators  $a \in \cB(H)$, in the presence of a tracial von Neumann algebra $\cM \subset  \cB(H)$:

\begin{equation*} \tag{\ref{suppscale2}}
\tau(s^\cM(a \xi)), \: \tau(s^\cM(a^* \xi)) \leq \lambda \cdot \tau(s^\cM(\xi)), \qquad \forall \xi \in H.
\end{equation*}

\begin{equation*} \tag{\ref{suppscale3}}
\tau(s_\ell^\cM(aq)), \: \tau(s_\ell^\cM(a^*q)) \leq \lambda \cdot \tau(q), \qquad \forall q \in \Proj(\cM).
\end{equation*}

We defined the associated \textit{vector} support expansion \cstar-algebra   as the closure of those $a \in \cB(H)$ that satisfy \eqref{suppscale2} for some $\lambda$, while the support expansion \cstar-algebra (which is a quantum uniform Roe algebra by Theorem \ref{thmprojsuppexp}) is the closure of those elements that satisfy \eqref{suppscale3} for some $\lambda$.  In this subsection we prove that conditions \eqref{suppscale2} and \eqref{suppscale3} are equivalent when $\cM$ is \textit{abelian}.  In general they are inequivalent, and we also show that there are vector support expansion \cstar-algebras with respect to certain $\cM\subseteq \cB(H)$ which are not quantum uniform Roe algebras in any way, meaning that there is no quantum coarse structure on $\cM\subseteq \cB(H)$ giving rise to this \cstar-algebra.

\begin{example} \label{2x2expansionctrex} We present a simple example showing that an operator may  satisfy \eqref{suppscale2} for some $\lambda$ but not satisfy   \eqref{suppscale3} for the same $\lambda$. For that, let $\cM = \mathbb{M}_2$ act standardly as left multiplication operators on $\mathrm{HS}_2$, the Hilbert space of $2\times2$ matrices with the Hilbert-Schmidt norm.  For any $\xi \in \mathrm{HS}_2$, $s^\cM(\xi)$ is left multiplication by the matrix projecting onto the range of $\xi$, and its nonnormalized trace is equal to the dimension of this range.  Let $a \in \cB(\mathrm{HS}_2)$ take $(\begin{smallmatrix} s & t\\ u & v\end{smallmatrix})$ to $(\begin{smallmatrix} s & t \\ u & v\end{smallmatrix})^t = (\begin{smallmatrix} s & u \\ t & v\end{smallmatrix})$.  For every $\xi \in \mathrm{HS}_2$, $\tau(s^\cM(T \xi) )= \tau(s^\cM(\xi))$ because the transpose map preserves rank (``row rank equals column rank").  Note also that $a$ is self-adjoint:
$$\langle \xi^t, \eta \rangle = \tau(\eta^* \xi^t) = \tau (\xi \eta^{t*}) = \tau(\eta^{t*} \xi) = \langle \xi, \eta^t \rangle.$$
Thus $a$ satisfies \eqref{suppscale2} for $\lambda=1$.

Let $q \in \cM$ be (left multiplication by) $(\begin{smallmatrix}1 & 0 \\ 0 & 0\end{smallmatrix})$.  Then $aq(\xi)$ is a vector of the form $(\begin{smallmatrix}s & 0 \\ t & 0 \end{smallmatrix})$, and $s^\cM_\ell(aq)$ is the $2 \times 2$ identity matrix (we are looking for the least projection such that left multiplication fixes all matrices of this form).  Thus $a$ does not satisfy \eqref{suppscale3} for $\lambda=1$.
\end{example}

This example is not so bad, because $a$ does satisfy \eqref{suppscale3} for $\lambda=2$.  Needing to double $\lambda$ is the worst that can happen for this $\cM \subset  \cB(H)$, so the vector support expansion \cstar-algebra still agrees with the support expansion \cstar-algebra.  In Section \ref{notqurasection} we give an infinite-dimensional version of this example, supplemented by additional analysis, showing that some vector support expansion \cstar-algebras are not quantum uniform Roe algebras at all.

\subsubsection{The two notions of expansion agree when $\cM$ is abelian}

\begin{theorem} \label{abelianexpansion}
Let $\cM \subset  \cB(H)$ be an abelian von Neumann algebra, and let $\tau$ be a faithful normal semifinite trace on $\cM$.  Then conditions \eqref{suppscale2} and \eqref{suppscale3} are equivalent for $a \in \cB(H)$, so that the vector support expansion \cstar-algebra equals the support expansion \cstar-algebra $\cstu(\cM,\mathscr{V}_\tau)$.
\end{theorem}

\begin{lemma} \label{singlevector}
Let $\cM \subset  \cB(H)$ be an abelian von Neumann algebra and $\xi_1, \dots ,\xi_n \in H$.  Then there are $c_1 =1, c_2, \dots ,c_n \in \mathbb{R}$ such that
$$\bigvee_{j=1}^n s^\cM(\xi_j) = s^\cM\left(\sum_{j=1}^n c_j \xi_j\right).$$
\end{lemma}

\begin{proof}
For $n=1$ there is nothing to show.  Take $n>1$, and assume we have proved the lemma for $n-1$.  Given $\xi_1, \dots , \xi_n$, apply the lemma to $\xi_1, \dots, \xi_{n-1}$ to find $c_2, \dots, c_{n-1}$ with $\vee_{j=1}^{n-1} s^\cM(\xi_j) = s^\cM(\sum_{j=1}^{n-1} c_j \xi_j)$.  For each $c \in \mathbb{R}$ let $$p_c = \left[\bigvee_{j=1}^n s^\cM(\xi_j)\right] - s^\cM\left[\left(\sum_{j=1}^{n-1}c_j \xi_j\right) + c \xi_n\right].$$

We show that the $p_c$ are pairwise orthogonal.  For simplicity rename $\sum_{j=1}^{n-1}c_j \xi_j$ as $\eta_1$ and $\xi_n$ as $\eta_2$.  We have
$$p_c = (s^\cM(\eta_1) \vee s^\cM(\eta_2)) - s^\cM(\eta_1 + c \eta_2)$$
so that $p_c(\eta_1 + c \eta_2) = 0$.  For $c \neq c'$ we compute
$$ (p_c  p_{c'})(\eta_1) + c (p_c  p_{c'})(\eta_2) = (p_c  p_{c'})(\eta_1 + c \eta _2) = 0 = (p_c  p_{c'})(\eta_1 + c' \eta_2) = (p_c  p_{c'})(\eta_1) + c' (p_c  p_{c'})(\eta_2).$$
As $c \neq c'$, we must have $(p_c  p_{c'})(\eta_1) = (p_c  p_{c'})(\eta_2) = 0$. 
It follows that $p_c  p_{c'}$ is perpendicular to $s^\cM(\eta_1) \vee s^\cM(\eta_2)$, of which it is a subprojection.  We conclude that $p_c  p_{c'}=0$.

Thus $\{p_c\mid c\in \R\}$ is an uncountable family of pairwise orthogonal projections.  For each $j$ all but countably many $p_c$ annihilate $\xi_j$, so there must be a $p_c$ that annihilates all $\xi_j$ and then is perpendicular to the projection $\vee_{j=1}^n s^\cM(\xi_j)$.  But $p_c$ lies beneath $\vee_{j=1}^n s^\cM(\xi_j)$ and so must be zero.  (To the conversant reader, we just used that cyclic projections are $\sigma$-finite, and the join of countably many $\sigma$-finite projections is $\sigma$-finite.)   The proof is completed by letting $c_n$ be any $c$ for which $p_c = 0$.  
\end{proof}

Note that the lemma is not true for general $\cM \subset  \cB(H)$: let $H$ be a Hilbert space containing linearly independent vectors $\xi_1, \xi_2$, and take $\cM = \cB(H)$.

\begin{proof}[Proof of Theorem \ref{abelianexpansion}]

First assume that $a \in \cB(H)$ satisfies \eqref{suppscale3}.  Then for any $\xi \in H$,
$$\lambda \tau(s^\cM(\xi)) \geq \tau(s^\cM_\ell(a s^\cM(\xi))) \geq \tau(s^\cM(a\xi)).$$
The first inequality is the assumption \eqref{suppscale3}.  For the second, note that $(a s^\cM(\xi))(\xi) = a\xi$, so $a \xi$ is in the range of $a s^\cM(\xi)$ and is therefore fixed by $s^\cM_\ell(a s^\cM(\xi))$.  The latter is a projection in $\cM$ and must therefore dominate $s^\cM(a\xi)$, which is the smallest $\cM$-projection fixing $a\xi$.  This shows that \eqref{suppscale2} holds (and commutativity of $\cM$ is not needed).

Next assume that $a \in \cB(H)$ satisfies \eqref{suppscale2}.  For any projection $q$ in $\cM$,
$$s^\cM_\ell(aq) = \bigvee_{\xi \in H} s^\cM(aq \xi) = \bigvee_{\xi \in qH} s^\cM(a\xi) = \lim_{\text{finite } F \subset  qH} \bigvee_{\xi \in F} s^\cM(a\xi).$$
For the last equality, we used that an infinite join is the weak limit of the increasing net of finite joins (the indices $F$ are finite subsets of $qH$, ordered by reverse inclusion).  By normality of $\tau$ we have
$$\tau(s^\cM_\ell(aq)) = \sup_{\text{finite }F \subset  qH} \tau\left(\bigvee_{\xi \in F} s^\cM(a\xi)\right).$$
Using Lemma \ref{singlevector} we can replace $\vee_{\xi \in F} s^\cM(a\xi)$ by the $\cM$-support of a linear combination of the $a\xi$, which is a single vector in $qH$:
$$\tau(s^\cM_\ell(aq)) = \sup_{\xi \in qH} \tau(s^\cM(a\xi)).$$
The desired conclusion follows by applying \eqref{suppscale2} to the term in the last supremum: for any $\xi \in qH$ we have
\[
\pushQED{\qed}  
\tau(s^\cM(a\xi)) \leq \lambda \tau(s^\cM(\xi)) \leq \lambda \tau(q). 
\qedhere
\]
\end{proof}

\subsubsection{A vector support expansion \cstar-algebra that is not a quantum uniform Roe algebra} \label{notqurasection}

In the proof of Theorem \ref{abelianexpansion}, we saw that \eqref{suppscale3} implies \eqref{suppscale2}, and Example \ref{2x2expansionctrex} shows that the converse can fail.  An infinite-dimensional version of this example, $a$ in the proof below, satisfies \eqref{suppscale3} for no $\lambda$, even though it satisfies \eqref{suppscale2} for $\lambda=1$.  This demonstrates that the sets of constrained operators and vector constrained operators can differ.

Establishing that a vector support expansion \cstar-algebra is not a quantum uniform Roe algebra is more complicated, because closures are involved, and the proof requires extra steps in order to handle approximations.  Note that we are not merely distinguishing the vector support expansion \cstar-algebra from the support expansion \cstar algebra; we are showing that it is not any quantum uniform Roe algebra at all.  The condition \eqref{suppscale3} plays no role here.

\begin{theorem} \label{notqura}
Let $\cM = \cB(\ell_2)$ be in standard form, i.e., acting by left multiplication on the Hilbert space $\mathrm{HS}$ of Hilbert-Schmidt operators on $\ell_2$.  The associated vector support expansion \cstar-algebra is not a quantum uniform Roe algebra on $\cM\subset  \cB(\mathrm{HS})$.  
\end{theorem}

\begin{proof}
We first repeat some  observations from Example \ref{2x2expansionctrex}, but in the infinite-dimensional setting.  Any element of $\cM$ is left multiplication by some $\ell_2$-operator $b$, denoted $L(b)$; similarly elements of $\cM'$ are right multiplications $R(b)$.  Given a Hilbert-Schmidt operator $c$ on $\ell_2$, we write $\hat{c}$ for the associated vector in $\mathrm{HS}$.  Then $s^\cM(\hat{c}) \in \cM$ is left multiplication by the projection onto the range of $a$.  Let $a \in \cB(\mathrm{HS})$ be the transpose map $\hat{c} \mapsto \widehat{c^t}$, which is a self-adjoint isometry on $\mathrm{HS}$.    Again $s^\cM(a(\hat{c})) = s^\cM(\widehat{c^t})$ has the same (possibly infinite) trace as $s^\cM(\hat{c})$, because the transpose operation preserves rank.  Thus $a$ satisfies \eqref{suppscale2} for $\lambda=1$.

Denote by $\cA$ the vector support expansion \cstar-algebra, which is the norm closure of the vector constrained operators.  Suppose towards a contradiction that $\cA$ is a quantum uniform Roe algebra, obtained as the norm closure of some algebraic quantum uniform Roe algebra $\cA_0$.  Since $a \in \cA$, there is $a_1 \in \cA_0$ with $\|a - a_1\| < \frac12$.  Now $\cA_0$ is the union of weak*-closed $\cM'$--$\cM'$ bimodules, so any element of $\overline{\mathrm{span}}^{w^*} \cM' a_1 \cM'$ also belongs to $\cA_0$.  The strategy of the proof is to build an element $a_2 \in \cA_0$ and show that it is not a limit of vector constrained operators.

Let $I_1 = \{1\}$, $I_2 = \{2,3\}$, $I_3=\{4,5,6\}$, etc.  Define $a_2 \in \cB(\mathrm{HS})$ by
\begin{align*}
a_2&=\mathrm{SOT}\text{-}\sum_{n\in\N} \frac{1}{\sqrt{n}}\sum_{i\in I_n} R(e_{1i})a_1 R(e_{ni}) \\ &= R(e_{11})a_1 R(e_{11}) + \frac{1}{\sqrt{2}}[R(e_{12}) a_1 R(e_{22}) + R(e_{13}) a_1 R(e_{23})] \\ &{} \quad + \frac{1}{\sqrt{3}}[R(e_{14}) a_1 R(e_{34}) + R(e_{15}) a_1 R(e_{35}) + R(e_{16}) a_1 R(e_{36})] + \dots
\end{align*}
Each term of the form $R(e_{1i})a_1 R(e_{ni})$ has range inside the $i$th column of $\mathrm{HS}$.  Each sum $\sum_{i\in I_n} R(e_{1i})a_1 R(e_{ni})$ is the sum of $n$ operators with norm $\leq \|a_1\|$ and pairwise orthogonal ranges and thus has norm $\leq \sqrt{n} \|a_1\|$.  This entails that $\frac{1}{\sqrt{n}}\sum_{i\in I_n} R(e_{1i})a_1 R(e_{ni})$ is an operator of norm $\leq \|a_1\|$ that annihilates vectors supported off the $n$th column of $\mathrm{HS}$ and has range inside the columns of $\mathrm{HS}$ indexed by $I_n$.  For different $n$ these have orthogonal left and right supports; it follows that the partial sums converge strongly to an operator $a_2$ of norm $\leq \|a_1\|$.  By the remark in the previous paragraph, we have $a_2 \in \cA_0$.

For convenience denote
$$a_3 = \mathrm{SOT}\text{-}\sum_{n\in\N} \frac{1}{\sqrt{n}}\sum_{i\in I_n} R(e_{1i})a R(e_{ni}),$$
so that
$$\|a_2 - a_3\| = \left\|\mathrm{SOT}\text{-}\sum_{n\in\N} \frac{1}{\sqrt{n}}\sum_{i\in I_n} R(e_{1i})(a_1 - a) R(e_{ni})\right\| \leq \|a_1 - a\| < \frac12$$
by the same argument as above.  The action of $a_3$ is nice:
$$\left( \begin{matrix}  t_{11} & t_{12} & t_{13} & \cdots \\ \vdots & \ddots & {} & {} \\ {} & {} & {} & {} \\ {} \\ {} & {} & {} & {} \end{matrix} \right)  \mapsto \left( \begin{matrix}  t_{11} & 0 & 0 & \cdots \\ 0 & \frac{t_{12}}{\sqrt{2}} & 0 & \cdots \\ 0 & 0 & \frac{t_{12}}{\sqrt{2}} & \cdots \\ \vdots & {} & {} & \ddots \end{matrix} \right).$$

To finish the proof we demonstrate that any $d \in \cB(\mathrm{HS})$ with $\|d - a_2\| < .1$ cannot satisfy \eqref{suppscale2} for any $\lambda > 0$.  Suppose there are such $d$ and $\lambda$.  Pick $n > 2 \lambda$.  By \eqref{suppscale2} at the vector $\widehat{e_{1n}}$,
$$\text{Tr}(s^\cM(d(\widehat{e_{1n}}))) \leq \lambda \text{Tr}(s^\cM(\widehat{e_{1n}})) = \lambda \text{Tr}(e_{11}) = \lambda.$$
Let $q$ be the projection in $\cB(\ell_2)$ such that $L(q)=s^\cM(d(\widehat{e_{1n}}))$, so that $q$ has rank $\leq \lambda$.  As $L(q^\perp)(d(\widehat{e_{1n}}))=0$, we have 
\begin{align*}
\|L(q^\perp)(
a_3(\widehat{e_{1n}}))\|_{\mathrm{HS}} &=\|L(q^\perp)((d-a_3)(\widehat{e_{1n}}))\|_{\mathrm{HS}}  \leq \|(d-a_3)(\widehat{e_{1n}})\|_{\mathrm{HS}} \\ &\leq \|d - a_3\| \leq \|d - a_2\| + \|a_2 - a_3\| < .1 + .5 = .6.
\end{align*}
On the other hand, $a_3(\widehat{e_{1n}})$ is $1/\sqrt{n}$ times a projection $r$ of rank $n$.  Since  $n > 2 \lambda$, Kaplansky's law for projections gives 
$$r - r \wedge q^\perp \sim r \vee q^\perp - q^\perp \leq q \quad \Rightarrow \quad \text{Tr}(r) - \text{Tr}(r \wedge q^\perp) \leq \text{Tr}(q) \quad \Rightarrow \quad n - \text{Tr}(r \wedge q^\perp) \leq \lambda < \frac{n}{2},$$
meaning $r \wedge q^\perp$ has rank at least $n/2$.  Finally
$$\|L(q^\perp)(
a_3(\widehat{e_{1n}}))\|_{\mathrm{HS}} = \left\|\left(\frac{\widehat{q^\perp r}}{\sqrt{n}}\right) \right\|_{\mathrm{HS}} \geq \left\|L(r \wedge q^\perp) \left(\frac{\widehat{q^\perp r}}{\sqrt{n}}\right) \right\|_{\mathrm{HS}} \geq \left\|\frac{\widehat{r \wedge q^\perp}}{\sqrt{n}} \right\|_{\mathrm{HS}} \geq  \frac{1}{\sqrt{2}},$$
which violates the previous inequality.
\end{proof}

\begin{remark}
We know from Section \ref{diffrep}  that the quantum uniform Roe algebras for the left multiplication representation $\cB(\ell_2) \subset  \cB(\mathrm{HS}) \simeq \cB(\ell_2 \otimes \ell_2)$ are given by amplifying quantum uniform Roe algebras for $\cB(\ell_2) \subset  \cB(\ell_2)$, and we know from Section \ref{B(H)qura} that the latter can be any \cstar-algebra.  So the point of Theorem \ref{notqura} is that the vector support expansion \cstar-algebra is \underline{not} of the form ``$(\text{\cstar-algebra}) \overline{\otimes} \cB(\ell_2)$", where the two factors are acting by left and right multiplication on $\mathrm{HS}$.

The mechanism of the proof makes sense from this perspective too: if the vector support expanstion \cstar-algebra were of this form inside $\cB(\mathrm{HS})$, it would be closed under strong limits and composition with right multiplications from $\cB(\ell_2)$, which is shown false.
\end{remark}

Even though vector support expansion \cstar-algebras do not always arise from a quantum coarse structure, we believe they are natural objects and interesting for study in their own right. 

\subsection{Rigidity and spatial implementation of *-isomorphisms}\label{SubsectionII1Factor}

A ``rigidity" theorem for uniform Roe algebras says that an equivalence at the operator algebra level (*-isomorphism or something weaker) implies coarse equivalence of metric/coarse spaces.  See for instance \cite{BaudierBragaFarahVignatiWillettFJA}.  
In this short subsection we notice that a key step in proving such results is true in a specific quantum situation (but not generally).

It is a basic fact of \cstar-theory that any *-isomorphism between represented \cstar-algebras containing the compact operators must be \emph{spatial}, i.e., of the form $\textrm{Ad}(u)$ for some unitary $u$ \cite[Theorem 1.3.4 and Corollaries to Theorem 1.4.4]{ArvesonInv}.  Classical uniform Roe algebras of connected coarse spaces (see Section \ref{SecConnected}) always contain the compact operators, and building on this, it was shown in \cite[Lemma 3.1]{BragaFarah2018} that all *-isomorphisms between classical uniform Roe algebras are spatial.  Spatiality is useful for proving rigidity-type results because the implementing unitary provides a foothold for building maps between the underlying sets.

But isomorphisms between quantum uniform Roe algebras do \emph{not} need to be spatially implemented, even when they arise from quantum coarse structures over the same represented von Neumann algebra.  For instance, the algebras
$$\left\{ \left( \begin{smallmatrix} a & 0 & 0 & 0 \\ 0 & a & 0 & 0 \\ 0 & 0 & b & 0 \\ 0 & 0 & 0 & b \end{smallmatrix} \right) \: \mid \: a,b \in \mathbb{C} \right\}, \qquad \left\{ \left( \begin{smallmatrix} a & 0 & 0 & 0 \\ 0 & a & 0 & 0 \\ 0 & 0 & a & 0 \\ 0 & 0 & 0 & b \end{smallmatrix} \right) \: \mid \: a,b \in \mathbb{C} \right\}$$
are quantum uniform Roe algebras for $\cM=\mathrm{M}_4(\C) \subset  \cB(\C^4)$ (Section \ref{quraMn}).  They are *-isomorphic to $\mathbb{C} \oplus \mathbb{C}$ and each other, but they are not spatially isomorphic.

We may still conclude spatiality for *-isomorphisms between quantum uniform  Roe algebras containing the compacts (as the algebras above fail to do).  The next result, which is about support expansion \cstar-algebras, points out a specific case of this.  Recall that a II$_1$-factor $\cM$ with trace $\tau$ is said to have \emph{property $\Gamma$} if for all $\eps>0$ and all $x_1,\ldots, x_n\in\cM$ there is a unitary $u\in \cM$ with $\tau(u)=0$  and $\|x_iu-ux_i\|_2 <\eps$ for all $i\in \{1.\ldots, n\}$.

\begin{proposition}\label{ThmII1notGammaIsoUnitaryImpl}
Let $(\cM, \tau) \subset  \cB(L_2(\cM, \tau))$ and $(\cN, \tau') \subset  \cB(L_2(\cN, \tau'))$ be  II$_1$-factors that do not have property $\Gamma$. Every *-isomorphism between the support expansion \cstar-algebras $\cstu(\cM,\mathscr{V}_{\tau})$ and $\cstu(\cN,\mathscr{V}_{\tau'})$   is spatially implemented.  
\end{proposition}

\begin{proof}
As $\cM$ does not have property $\Gamma$, \cite[Theorem 2.1]{Connes1976Annals} says that $\cK(L_2(\cM))\subset  \mathrm{C}^*(\cM \cup \cM')$.  But the latter is contained in $\cstu(\cM,\mathscr{V}_{\tau})$ by Theorem \ref{PropDimContMinside}(1).  Similarly $\cK(L_2(\cN))\subset  \cstu(\cN,\mathscr{V}_{\tau'})$, so by the preceding remarks all isomorphisms between these algebras are spatial.
\end{proof}

Developing rigidity for quantum uniform Roe algebras would be an interesting direction for future research.

\section{Morphisms in the quantum category}\label{SectionQuantumCoarseEq}
 
The current section discusses morphisms between quantum  coarse spaces.  We consider individual morphisms, various equivalences, and a notion of  subspace, and we prove additional results in case the quantum coarse space is metrizable. In Section \ref{SubsectionAsyDim} we quantize the concept of  asymptotic dimension and show that it is stable under quantum coarse embeddings (Theorem \ref{ThmAsyDim}).  

The next terminology  was introduced in  \cite{Kornell2011} (see also \cite{ChavezDominguezSwift2020}):

\begin{definition}
 A unital weak$^*$-continuous $^*$-homomorphism   $\varphi:\cM\to \cN$ between von Neumann algebras is called a \emph{quantum function}.
\end{definition}

Quantum functions are the quantum versions of ordinary functions  $X\to Y$. Indeed, given sets $X$ and $Y$, any map  $f:X\to Y$ canonically induces a quantum function   $\varphi_f:\ell_\infty(Y)\to \ell_\infty(X)$ by letting $\varphi_f(g)=g\circ f$ for all $g\in \ell_\infty$. Moreover, any quantum function $\varphi: \ell_\infty(Y)\to \ell_\infty(X)$ induces a map $f_\varphi:X\to Y$ by letting $f(x)=y$ if $\chi_{\{x\}}\leq \varphi(\chi_{\{y\}})$. These two constructions are clearly inverse to each other.

Quantum functions behave well with respect to pullbacks of intrinsic quantum relations. Precisely, given intrinsic quantum coarse spaces $(\cM,\mathscr{R})$ and $(\cN,\mathscr{Q})$,  every such  $\varphi:\cM\to \cN$ induces a canonical map  $\varphi^*:\mathrm{IQRel}(\cN)\to \mathrm{IQRel}(\cM)$ by letting 
\[\varphi^*(\cQ)=\Big\{(p,q)\in \Proj(\cM\bar\otimes \cB(\ell_2))^2: \,  \big((\varphi\otimes 1)(p),(\varphi\otimes 1)(q)\big)\in \cQ\Big\} = (\varphi \otimes 1)^{-1}(\cQ)\]
for all $\cQ\in \mathscr{Q}$ (see \cite[Proposition 2.25]{Weaver2012MemoirAMS} for a proof that this map is well-defined).

The following is the quantum version of   coarse maps.\footnote{ Recall, if $(X,\cE)$ and $(Y,\cF)$ are coarse spaces, a map  $f:X\to Y$ is \emph{coarse} if   $(f\times f)[\cE]\subset  \cF$.} 

\begin{definition}\label{DefiQuantumCoarse}
Let $(\cM,\mathscr{R}) $ and  $(\cN,\mathscr{Q}) $ be   intrinsic quantum coarse spaces. We call a quantum function  $\varphi:\cM\to \cN$  \emph{quantum coarse} if $\varphi^*[\mathscr{Q}]\subset   \mathscr{R}$. If $(\cM \subset  \cB(H),\mathscr{V})$ and $(\cN \subset  \cB(K),\mathscr{U})$ are quantum coarse spaces, we define   quantum coarse maps  considering the intrinsic quantum coarse spaces   $\mathscr{R}=\mathscr{R}_\mathscr{V}$ and $\mathscr{Q}=\mathscr{R}_\mathscr{U}$.
\end{definition}

 We start by showing that quantum coarseness is equivalent to coarseness in the classical setting:

\begin{proposition}\label{PropQuantumCoarseIFFCoarse}
Let $(X,\cE)$ and $(Y,\cF)$ be coarse spaces and consider $\ell_\infty(X)$ and $\ell_\infty(Y)$ endowed with the quantum coarse structures induced by $\cE$ and $\cF$, respectively. A map   $f:X\to Y$   is coarse if and only if $\varphi_f:\ell_\infty(Y)\to \ell_\infty(X) $ is quantum coarse. 
\end{proposition}

\begin{proof}
Let $E\subset  X^2$ and let $\cR_E=\cR_{\cV_E}$, i.e., $(p,q)\in \cR_E$ if and only if $\exists(x,y)\in E$ so that $p(e_{xy}\otimes 1)q\neq 0$. Then we have that 
\begin{align*}
    \varphi^*_f(\cR_E)&=\{(r,s): \,  \exists (x,y)\in E,\ [(\varphi_f\otimes 1)(r)](e_{xy}\otimes  1)[(\varphi_f\otimes 1)(s)]\neq 0\}\\
    &=\{(r,s): \,  \exists (x,y)\in E, \   r(e_{f(x)f(y)}\otimes 1) s\neq 0\}\\
    &=\cR_{ f\times f(E)}
\end{align*}
(the second equality above holds since $\varphi_f$ is weak$^*$ continuous, and $r,s$ are weak* limits of linear combinations of simple tensors).
The result then follows from   Proposition \ref{PropositionWeaverRelation} and Theorem \ref{ThmWeaverQRelAndIntQRel}.
\end{proof}

If $f:X\to Y$ is an injective coarse map between coarse spaces $(X,\cE)$ and $(Y,\cF)$, then the isometric embedding   $u_f:\ell_2(X)\to \ell_2(Y)$ given by $u_f(\delta_x)=\delta_{f(x)}$, for all $x\in Y$, induces an embedding $\mathrm{Ad}(u_f):\cstu(X)\to \cstu(Y)$ (see \cite[Theorem 1.2]{BragaFarahVignati2019} for details). Moreover, if $f$ is a bijective coarse equivalence, this embedding is an isomorphism (see  \cite[Theorem 8.1]{BragaFarah2018} for details). For this reason, we are interested in understanding the quantum versions of ``injective coarse maps'' and ``bijective coarse equivalences''. As $f:X\to Y$ is injective if and only if $\varphi_f:\ell_\infty(Y)\to \ell_\infty(X)$ is surjective,  in the quantum world the notion of an  injective coarse map $f:X\to Y$ is replaced by a surjective quantum coarse function $\cM\to \cN$. Similarly, a bijective coarse equivalence $X\to Y$ becomes an isomorphism $\varphi:\cM\to \cN$ with both $\varphi$ and $\varphi^{-1}$ being quantum coarse.  

\begin{definition}\label{DefiQuantumCoarseIso}
Let $(\cM \subset  \cB(H_\cM),\mathscr{V})$ and $(\cN \subset  \cB(H_\cN),\mathscr{U})$ be quantum coarse spaces. An isomorphism $\varphi:\cM\to \cN$ is called a \emph{quantum coarse isomorphism} if both $\varphi$ and $\varphi^{-1}$ are quantum coarse functions.    
\end{definition}

A quantum coarse isomorphism is just a change of representation for a single intrinsic quantum coarse space, so  the next theorem follows from Sections \ref{diffrep} and \ref{subsectioniqura}.

\begin{theorem}\label{ThmEmbAndIso}
Let $\varphi:\cM\to \cN$ be a quantum coarse isomorphism between the quantum coarse spaces $(\cM \subset  \cB(H_\cM),\mathscr{V})$ and $(\cN \subset  \cB(H_\cN),\mathscr{U})$.
\begin{enumerate}
    \item\label{ThmEmbAndIso1} If $\varphi$ is spatially implemented, then it induces a spatial isomorphism  $\cstu(\cM ,\mathscr{V} ) \simeq \cstu(\cN,\mathscr{U})$.
    \item \label{ThmEmbAndIso2} In any case there is a spatially implemented isomorphism \[\cstu(\cM,\mathscr{V}) \bar\otimes \cB(K) = \cstu(\cM \otimes 1_K,\mathscr{V} \bar\otimes \cB(K))  \, \simeq \, \cstu(\cN\otimes 1_K,\mathscr{U} \bar{\otimes} \cB(K)) = \cstu(\cN,\mathscr{U})\bar\otimes \cB(K)\] 
\end{enumerate}
for any Hilbert space $K$ with $\dim K \geq \max\{ \aleph_0, \dim H_\cM, \dim H_\cN\}$.
\end{theorem}

For a simple example of the difference between the two parts of Theorem \ref{ThmEmbAndIso}, let $\cM=\cN = \C$, which carries a unique intrinsic quantum coarse structure.  The identity $\varphi\colon \cM\to \cN$ is a quantum coarse isomorphism.  If we represent $\cM$ in $\cB(\C)$  and $\cN$ in $\cB(\C^2)$ as multiples of the identity, their quantum uniform Roe algebras are $\C$ and $\cB(\C^2)$, respectively, which are not (spatially) isomorphic until after amplification.  
 
We now prove the quantum version of the discrete result about embeddability.  

\begin{theorem}\label{ThmEmbAndIsoF}
Let $(\cM,\mathscr{V})$ and $(\cN,\mathscr{U})$ be quantum coarse spaces and let $\varphi:\cM\to \cN$ be a spatially implemented surjective quantum function.  If $\varphi$ is  quantum coarse, then $\cstu(\cN,\mathscr{U})$ spatially embeds into $\cstu(\cM ,\mathscr{V} )$.
\end{theorem}

\begin{proof}
 Let $\mathscr{R}=\mathscr{R}_\mathscr{V}$ and $\mathscr{Q}=\mathscr{R}_\mathscr{U}$.  As $\varphi$ is a quantum function, there is a central projection $r\in \cM$ such that $\ker(\varphi)=(1-r)\cM$. Hence, as $\varphi$ is spatially implemented, there is a surjective isometry   $u:H_\cN\to \mathrm{Im}(r)$ such that $\varphi(a)=u^*au$ for all $a\in \cM$.  Let $\psi:\cB(H_\cN )\to \cB(H_\cM)$ be the spatial embedding given by   $\psi(b)=ubu^*$ for all $b\in \cB(H_\cN)$.

\begin{claim}\label{Claim111}
Let $\cQ\in \mathscr{Q}$. Then   $\psi(b)\in\cV_{\varphi^*(\cQ)}$ for all     $b\in\cV_{\cQ}$.
\end{claim}

\begin{proof}
Fix $b\in \cV_{\cQ}$ and projections $p,q\in \cM \bar\otimes \cB(\ell_2)$ with $(p,q)\not\in \varphi^*(\cQ)$. So, $((\varphi\otimes 1)(p),(\varphi\otimes 1)(q))\not\in \cQ$ and we have \[[(\varphi\otimes 1)(p)](b\otimes 1)[(\varphi\otimes 1)(q)]=0.\]
Therefore, as $\varphi$ is weak$^*$-continuous, approximating  $p$ and $q$ by elements in $\cM \otimes \cB(\ell_2)$, we have that \[0 = (u^* \otimes 1)p(u \otimes 1)(b \otimes 1)(u^* \otimes 1)q(u \otimes 1)\] and then
 $p(\psi(b)\otimes 1)q=0$. The arbitrariness of $p$ and $q$ implies that $\psi(b)\in \cV_{\varphi^*(\cQ)}$.
\end{proof}

The previous claim shows that $\psi$ takes $\bigcup_{\cQ\in \mathscr{Q}}\cV_{\cQ}$ into  $\bigcup_{\cQ\in \mathscr{Q}}\cV_{\varphi^*(\cQ)}$. As $\varphi$ is quantum coarse, $\varphi^*(\cQ)\in \mathscr{R}$ for all $\cQ\in \mathscr{Q}$, so $\psi$ takes $\bigcup_{\cQ\in \mathscr{Q}}\cV_{\cQ}$ into  $\bigcup_{\cR\in \mathscr{R}}\cV_{\cR}$. Hence, $\psi\restriction\cstu(\cN,\mathscr{U})$ is an embedding of   $\cstu(\cN,\mathscr{U})$ into $\cstu(\cM ,\mathscr{V})$.
\end{proof}

Although a quantum coarse isomorphism $\cM\to\cN$ induces an isomorphism between quantum uniform Roe algebras, this is not the quantum version of \emph{coarse equivalence}, since coarse equivalences do not need to be bijective.   In order to deal with a ``nonbijective quantum coarse isomorphism'',   we must first  introduce the notion of \emph{closeness} in the quantum setting.  Recall, if $(Y,\cF)$ is a coarse space and $X$ is a set, then maps $f,g:X\to Y$ are \emph{close} if there is $F\in \cF$ so that $(f(x),f(y))\in F$ for all $x\in X$.

\begin{definition}
Let $ \cM$ and $\cN$ be  von Neumann algebras and let $\mathscr{R}$ be an intrinsic quantum coarse structure on $\cM$. Quantum functions $\varphi,\psi:\cM\to \cN$ are called \emph{quantum close} if there is $\cR\in \mathscr{R}$ so that $(p,q)\in \cR$ for all $p,q\in \Proj(\cM\bar\otimes \cB(\ell_2))$ with $(\varphi\otimes 1)(p)(\psi\otimes 1)(q)\neq 0$.
\end{definition}

\begin{proposition}
Let $X$ be a set and $(Y,\cF)$ be a coarse space, and consider  $\ell_\infty(Y)$ endowed with the quantum coarse structure induced by   $\cF$. Maps $f,g:X\to Y$ are close if and only if $\varphi_f,\psi_g\colon \ell_\infty(Y)\to \ell_\infty(X)$ are quantum close.
\end{proposition}

\begin{proof}
Suppose  $\varphi_f$  and $\varphi_g$ are quantum close, and let $F\in \cF$ be such that $\cR_{F}$ witnesses that those maps are quantum close.  Then, for each $x\in X$, we have
\[[(\varphi_f\otimes 1)(e_{f(x)f(x)}\otimes 1)][(\varphi_g\otimes 1)(e_{g(x)g(x)}\otimes 1)]= e_{xx}\otimes 1,\]
which implies that $((e_{f(x)f(x)}\otimes 1),(e_{g(x)g(x)}\otimes 1))\in \cR_{F}$. By the definition of $\cR_E$, this implies that $(f(x),g(x))\in F$.

Suppose now that $f$ and $g$ are close, and say $F\in \cF$ is so that $(f(x),g(x))\in F$ for all $x\in X$. Let $\cR_F$ be the intrinsic quantum relation  on $\ell_\infty(Y)$ given by $F$. If $p,q\in \Proj(\ell_\infty(Y)\bar\otimes \cB(\ell_2))$ are such that  $[(\varphi_f\otimes 1)(p)][(\varphi_g\otimes 1)(q)]\neq 0$, pick $x\in X$ so that  $(\varphi_f\otimes 1)(p)(e_{xx}\otimes 1)(\varphi_g\otimes 1)(q)\neq 0$. Then, as $\varphi_f$ and $\varphi_g$ are  weak$^*$-continuous, we must have that $p(e_{f(x)g(x)}\otimes 1)q\neq 0$, i.e., $(p,q)\in \cR_F$.
\end{proof}

We now introduce the quantum version of coarse equivalence:

\begin{definition}\label{DefiQuantumCoarseEquivalence}
Let $(\cM,\mathscr{V})$ and $(\cN,\mathscr{U})$ be quantum coarse spaces and $\varphi:\cM\to \cN$ be a quantum function. We say that $\varphi$ is a \emph{quantum coarse equivalence} if $\varphi$ is quantum coarse and there is a quantum coarse map $\psi:\cN\to \cM$ so that $\psi\circ \varphi$ and $\varphi\circ \psi$ are quantum close to the identities of $\cM$ and $\cN$, respectively. In this case  $(\cM,\mathscr{V})$ and $(\cN,\mathscr{U})$ are called \emph{quantum coarsely equivalent} and $\psi$ is a \emph{quantum coarse inverse} of $\varphi$.
\end{definition}

Clearly, if $\varphi:\cM\to \cN$ is a quantum coarse isomorphism, them $\varphi$ is a quantum coarse equivalence. We return to quantum equivalences in Subsections \ref{SubsectionSubspace} and  \ref{SubsectionAsyDim}.

\subsection{Quantum coarse subspaces}\label{SubsectionSubspace}
Besides equivalences between quantum coarse spaces, we want to be able to talk about embeddings. For that, we must first deal with quantum \emph{subspaces}. If $\cM$ is a von Neumann algebra   and $r\in \cM$ is a central projection, then $\cN=r\cM$ can be be seen as a von Neumann algebra in $\cB(K)$, where $K=\mathrm{Im}(r)$.  If $\mathscr{V}$ is a quantum coarse structure on $\cM$, we let  \[\mathscr{V}_\cN=\{r\cV r: \,  \cV\in \mathscr{V}\}.\footnote{As the elements in $\mathscr{V}$ are bimodules over $\cM'$,  we have that $\mathscr{V}_\cN\subset  \mathscr{V}$.}\]  Considering  $\mathscr{V}_\cN$ as a family of quantum relations on  $\cN\subset  \cB(K)$,  $\mathscr{V}_\cN$ is a quantum coarse structure on $\cN$ and we call $(\cN,\mathscr{V}_\cN)$ a \emph{quantum coarse subspace of $(\cM,\mathscr{V})$} (cf.\ \cite[Definition 2.35]{KuperbergWeaver2012MemAMS}).
 We denote the intrinsic quantum coarse structure $\mathscr{R}_{\mathscr{V}_\cN}$ on $\cN$ by $\mathscr{R}_\cN$.
 
We now show that quantum coarseness is ``independent of subspaces". In the classical case, a map $f:X\to Y$ between coarse spaces is coarse if and only if $f:X\to Z$ is coarse for (any) coarse space $Z$ with $Y\subset  Z$. The next two propositions show that quantum coarse maps satisfy the same property.

 \begin{proposition}\label{PropSubspaceEmb}
Let $(\cM,\mathscr{V})$  be a quantum coarse space and let $r$ be a central projection in $\cM$. The map $\pi:a\in\cM\mapsto ra\in r\cM$ is  quantum coarse.
\end{proposition}

\begin{proof}
Let $\cN=r\cM$. Say $\cQ\in \mathscr{R}_\cN$. So there is $\cV\in \mathscr{V}$ so that, thinking of $r \cV r$ as being in  $\mathscr{V}_\cN$,  $\cQ=\cR_{r\cV r}$. Hence, 
\begin{align*}
\pi^*(\mathcal Q)&=\{(p,q)\mid ((\pi\otimes 1)(p),(\pi\otimes 1)(q))\in \mathcal \cR_{r\cV r}\}\\
&=\{(p,q)\mid \exists v\in \cV,\ [(\pi\otimes 1)(p)(v\otimes 1)][(\pi\otimes 1)(q)]\neq 0\}\\
&=\{(p,q)\mid \exists v\in \cV,\  (p)(rvr\otimes 1)(q)\neq 0\}\\
&=\{(p,q)\mid \exists v\in r\cV r,\  (p)(v\otimes 1)(q)\neq 0\}\\
&=\cR_{r\cV r}.
\end{align*}
So,   $\pi^*(Q)=\cR_{r\cV r}\in  \mathscr{R}_{\mathscr{V}}$, i.e., $\pi$ is quantum coarse.
\end{proof}

 If $\varphi:\cM\to \cN$ is a quantum function between von Neumann algebras, then $\ker(\varphi)$ is a weak$^*$-closed ideal. Hence, there is a central projection $r\in\cM$ such that $\ker(\varphi)=(1-r)\cM$. If $\mathscr{V}$ is a quantum coarse structure on $\cM$, then we consider $\cM/\ker(\varphi)$ as a quantum coarse space endowed with the quantum coarse structure given by the canonical isomorphism $\cM/\ker(\varphi)\cong r\cM$. 

 \begin{proposition}\label{PropTakeQuotCoarse}
Let $(\cM,\mathscr{V})$ and $(\cN,\mathscr{U})$ be quantum coarse spaces. Let $\varphi:\cM\to \cN$ be a quantum function and $\psi:\cM/\ker(\varphi)\to \cN$ be the map induced by $\varphi$. Then $\varphi$ is quantum coarse   if and only if $\psi$   is  quantum coarse.
\end{proposition}

\begin{proof} 
Let $r$ be a central projection in $\cM$ so that $(1-r)\cM=\ker(\varphi)$, so we identify $\cM/\ker(\varphi)$ with $r\cM$.   Let $\pi:\cM\to r\cM$ be the map given by $\pi(a)=r a$ for all $a\in\cM$; so $\pi$ is quantum coarse  (Proposition \ref{PropSubspaceEmb}) and $\varphi=\psi\circ \pi$. In particular, if $\psi$ is quantum coarse, so is $\varphi$. Moreover,  if $\cQ\in \mathrm{IQRel}(\cN)$, then $\psi^*(\cQ)\subset  \varphi^*(\cQ)$. Therefore, if $\varphi$ is quantum coarse, so is $\psi$.
\end{proof}

 In the discrete setting, a map $f:X\to Y$ is a coarse embedding if $f:X\to f(X)$ is a coarse equivalence. We then make the following definition:
 
\begin{definition}\label{DefiQuantumCoarseEmbedding}
 Let $(\cM,\mathscr{V})$ and $(\cN,\mathscr{U})$ be quantum coarse spaces. A quantum function   $\varphi:\cM\to \cN$ is called a \emph{quantum coarse embedding} of $(\cN,\mathscr{U})$ into $(\cM,\mathscr{V})$ if the induced map $\psi:\cM/\ker(\varphi)\to \cN$ is a quantum coarse equivalence.
 \end{definition}

\subsection{Expanding maps and coboundedness}
Notoriously, there are two equivalent ways of defining equivalences  between  coarse spaces:  (1)  a coarse map $f:X\to Y$ is a coarse equivalence if there is a coarse map $g:Y\to X$ so that $g\circ f$ and $f\circ g$ are close to $\mathrm{Id}_X$ and $\mathrm{Id}_Y$, respectively, and (2) a coarse map is a coarse equivalence if it is \emph{expanding} and \emph{cobounded}.\footnote{A map $f:X\to Y$ between coarse spaces $(X,\cE)$ and $(Y,\cF)$  is \emph{expanding} if $(f\times f)^{-1}[\cF]\subset  \cE$ and \emph{cobounded} if there is $F\in \cF$ so that for all $y\in Y$ there is $x\in X $
 with $(y,f(x))\in F$.} The first one, being clearly a symmetric property, is much more natural and therefore it was the one quantized in the previous section. 
 
 As the existence of quantum functions  between von Neumann algebras is not immediate, the reader should not expect that quantum coarse equivalence should be equivalent to the existence of a single quantum coarse function $\cM\to \cN$ which is also ``quantum cobounded'' and ``quantum expanding'' (see Remark \ref{RemarkBack}). However, as we see in this section, there are still natural generalizations of expansion and coboundedness which are indeed implied by coarse equivalences (see Theorem \ref{ThmEmbImpliesExpCob}). Moreover, as we see in Theorem \ref{ThmEmbAndIsoHERE}, those concepts are related to embeddings of quantum uniform Roe algebras into \emph{hereditary} subalgebras.
 
\begin{definition}\label{DefiQuantumCoarseExpanding}
Let $(\cM,\mathscr{R}) $ and  $(\cN,\mathscr{Q}) $ be   intrinsic quantum coarse spaces. We call a quantum function   $\varphi:\cM\to \cN$   \emph{quantum expanding} if $(\varphi^*)^{-1}[\mathscr{R}]\subset  \mathscr{Q}$. If $(\cM,\mathscr{V})$ and $(\cN,\mathscr{U})$ are quantum coarse spaces, we define   quantum expanding maps by considering   the intrinsic quantum relations   $\mathscr{R}=\mathscr{R}_\mathscr{V}$ and $\mathscr{Q}=\mathscr{R}_\mathscr{U}$.
\end{definition}
 
As usual, we start by noticing that this quantization of expanding functions indeed coincides with the usual notion in the classical setting.

\begin{proposition}\label{PropQuantumExpIFFExp}
Let $(X,\cE)$ and $(Y,\cF)$ be coarse spaces and consider $\ell_\infty(X)$ and $\ell_\infty(Y)$ endowed with the quantum coarse structures induced by $\cE$ and $\cF$, respectively. A map   $f:X\to Y$   is expanding   if and only if $\varphi_f :\ell_\infty(Y)\to \ell_\infty(X)$ is quantum expanding. 
\end{proposition}

\begin{proof}
Let   $\overline{f\times f}:\cP(X^2)\to \cP(Y^2)$ be the    map given by $\overline{f\times f}(E)=(f\times f)[E]$ for all $E\in \cP(X^2)$, and notice that  $f$ is expanding  if and only if $\overline{f\times f}^{-1}[\cE]\subset  \cF$.  Given $F\subset  Y^2$, let $\cR_F=\cR_{\cV_F}$ and notice that, as $\varphi_f$ is weak$^*$-continuous, we have  that 
\begin{align*}
    \cQ\in (\varphi_f^*)^{-1}(\cR_{F})\ &\iff\  \exists E\in \cP(X^2)\text{ with }\cQ=\cR_E \text{ and }  \varphi_f^*(\cR_E )=\cR_F\\
    &\iff\ \exists E\in \cP(X^2)\text{ with }\cQ=\cR_E \text{ and } \\
    &\ \ \ \ \ \{(r,s): \,  \exists (x,y)\in E,\  r (e_{f(x)f(y)}\otimes 1)s\neq 0\}\\
    &\ \ \ \ \ \ \ \ = \{(r,s): \,  \exists (z,w)\in F,\  r(e_{zw}\otimes 1)s\neq 0\}\\
    &\iff\ \exists E\in \cP(X^2)\text{ with }\cQ=\cR_E \text{ and }\overline{f\times f}(E)=F\\
    &\iff\ \cQ\in \{\cR_E: \,  E\in  \overline{f\times f}^{-1}[\{F\}]\}.
\end{align*}
So $(\varphi_f^*)^{-1}(\cR_{F})=\{\cR_E: \,  E\in  \overline{f\times f}^{-1}[\{F\}]\}$ for all $F\subset  Y^2$ and the result follows.
\end{proof}
 
If $X$, $Y$, and $Z$ are coarse spaces and $Y\subset  Z$, then a map $f:X\to Y$ is expanding if and only if $f:X\to Z$ is expanding. The goal of the next two propositions is to show that the same holds for quantum expansion. 

\begin{proposition}\label{PropPIExp}
Let $(\cM,\mathscr{V})$  be a quantum coarse space and let $r$ be a central projection in $\cM \subset  \cB(H)$. The map $\pi:a\in\cM\mapsto ra\in r\cM$ is   quantum expanding.
\end{proposition}

\begin{proof}
Say $\cR\in \mathscr{R}$ and pick $\cQ\in \mathrm{IQRel}(r\cM)$ so that $\varphi^*(\cQ)=\cR$. By Theorem \ref{ThmWeaverQRelAndIntQRel}, there are  quantum relations $\cV$ and $\cU$ on $\cM$ so that $\cR=\cR_{\cV}$ and, thinking of $r\cU r$ as a quantum relation on $r\cM$, $\cQ=\cR_{r \cU r}$. Unfolding definitions, it is straightforward to check that, considering $r\cV r$ as a subspace of $\cB(r H)$, we have $\cQ\subset  \cR_{r\cV r}$. So $\cQ\in \mathscr{R}_{r\cM}$, and $\pi$ is quantum expanding.
\end{proof}

\begin{proposition}\label{PropSubspaceEquiv}
Let $(\cM,\mathscr{V})$ and $(\cN,\mathscr{U})$ be quantum coarse spaces. Let $\varphi:\cM\to \cN$ be a quantum function and $\psi:\cM/\ker(\varphi)\to \cN$ the map induced by $\varphi$. Then $\varphi$ is quantum expanding   if and only if $\psi$   is  quantum expanding.
\end{proposition}

\begin{proof} 
Let $r$ be a central projection in $\cM$ so that $(1-r)\cM=\ker(\varphi)$,   and identify $\cM/\ker(\varphi)$ with $r\cM$. For simplicity,   let $\mathscr{R}=\mathscr{R}_\mathscr{V}$,   $\mathscr{Q}=\mathscr{R}_\mathscr{U}$, and  $\mathscr{R}'=\mathscr{R}_{{r\cM}}$. Let $\pi:\cM\to r\cM$ be the map given by $\pi(a)=r a$ for all $a\in\cM$, so  $\pi$ is expanding (Proposition \ref{PropPIExp}) and $\varphi=\psi\circ \pi$.

As $\varphi=\psi\circ \pi$, if $\psi$ is quantum expanding, so is $\varphi$. Suppose now that $\varphi$ is quantum expanding. Let $\cR\in \mathscr{R}'$ and pick $\cQ\in \mathrm{IQRel}(\cN)$ with $\psi^*(\cQ)=\cR$. By Theorem \ref{ThmWeaverQRelAndIntQRel}, there is   $\cV\in \mathscr{V}$ so that, thinking of $r\cV r$ as a quantum relation on $r \cM$,  we have $\cR=\cR_{ r\cV r}$. Let $\cV'=r\cV r$ be considered as a subspace of $\cB(H_\cM)$. Then, as $\psi^*(\cQ)=\cR_{ r\cV r}$, we have that
\begin{align*}
  \varphi^*(\cQ)&=\{(p,q): \,  ([(\psi\otimes 1)(r\otimes 1)](p), [(\psi\otimes 1)(r\otimes 1)](q))\in \cQ\}\\
  &=\{(p,q): \,  ( (r\otimes 1)(p), (r\otimes 1)(q))\in \psi^*(\cQ)\}\\
  &=\{(p,q): \,  \exists a\in \cV,\ p(rar\otimes 1)  q\neq 0\}\\
  &=\cR_{\cV'}.
\end{align*} 
As $\varphi$ is quantum expanding and $ \cV'\in \mathscr{V}$, it follows that $\cQ\in \mathscr{Q}$. So, $\psi$ is quantum expanding.
\end{proof}

\begin{definition} \label{DefQuantumCoarseCobounded}
Let $(\cM,\mathscr{R})$   be an intrinsic quantum coarse space and $\cN$ be a von Neumann algebra. A quantum map   $\varphi:\cM\to \cN$ is called \emph{quantum cobounded} if and only if there is $\cR\in \mathscr{R}$ so that for all $p\in \Proj_*(\cM\bar\otimes \cB(\ell_2) )$ there is $q\in \Proj(\cM \bar\otimes \cB(\ell_2)) $ such that   $(p ,(r \otimes 1)q)\in \cR$, where $r$ is the  central projection in $\cM$ so that $\ker(\varphi)=(1-r)\cM$.
\end{definition}

The proof of the next proposition is straightforward,  so we omit it.

\begin{proposition}
Let $X$ be a set, $(Y,\cF)$ be a coarse space and consider  $\ell_\infty(Y)$ endowed with the quantum coarse structure induced by   $\cF$. A map $f:X\to Y$ is cobounded  if and only if $\varphi_f:\ell_\infty(Y)\to \ell_\infty(X)$  is quantum cobounded.\qed
\end{proposition}

\begin{proposition}\label{Prop19}
Let $(\cM,\mathscr{V})$ and $(\cN,\mathscr{U})$ be quantum coarse spaces, and $\varphi:\cM\to \cN$ and $\psi:\cN\to \cM$ be quantum functions.
\begin{enumerate}
    \item \label{Item1q} If   $\psi\circ \varphi$ is quantum close to the identity $\cM\to \cM$, then $\varphi$ is quantum cobounded.
\item \label{Item2q} If $\varphi\circ\psi$ is quantum close to the identity $\cN\to \cN$ and $\psi$ is quantum coarse, then $\varphi$ is quantum expanding.
\end{enumerate}
\end{proposition}

\begin{proof}
Let $\mathscr{R}=\mathscr{R}_\mathscr{V}$ and $\mathscr{Q}=\mathscr{R}_\mathscr{U}$. 

\eqref{Item1q}   Fix   $\cR\in \mathscr{R} $ witnessing that $\psi\circ \varphi$ is quantum close to the identity $\cM\to \cM$. Let  $r\in \cM$ be the central projection so that $\ker(\varphi)=(1-r)\cM$ and fix a nonzero projection $p\in \cM\bar\otimes\cB(\ell_2)$. Let \[q_0=1_{H_{\cN}\otimes \ell_2}-\bigvee\{q'\in  \Proj(\cN\bar\otimes\cB(\ell_2) ): \,  p[(\psi\otimes 1)(q')]=0\};\]   so $q_0 \in  \cN\bar\otimes\cB(\ell_2)$ and $p\leq (\psi\otimes 1)(q_0 )$.   Similarly,   let \[p_0=1_{H_{\cM}\otimes \ell_2}-\bigvee\{p'\in \Proj( \cM\bar\otimes\cB(\ell_2)): \,  q_0  [(\varphi\otimes 1)(p')]=0\}.\] So  $ p_0 \in \cM\bar\otimes\cB(\ell_2)$ and $q_0 \leq (\varphi\otimes 1)(p_0)=[(\varphi\otimes 1)(r\otimes 1)](p_0)$. Therefore, we have that  $(\psi\otimes 1)(q_0 )\leq [(\psi\otimes 1)(\varphi\otimes 1) (r\otimes 1)](p_0)$ and, as  $p\leq (\psi\otimes 1)(q_0 )$, we conclude that  \[p[(\psi\otimes 1)(\varphi\otimes 1)(r\otimes 1)]p_0\neq 0.\] By our choice of $\cR$,  it must follow that $(p ,(r\otimes 1)p_0)\in \cR$. So $\varphi$ is cobounded.  

\eqref{Item2q}   Fix   $\cQ\in \mathscr{Q}$ witnessing that $  \varphi\circ\psi$ is quantum close to the identity $\cN\to \cN$. Let  $\cR\in \mathscr{R} $ and pick  $\cQ'\in (\varphi^*)^{-1}[\{\cR\}]$. Let us show that $\cQ'\subset  \cQ\circ\psi^*(\cR) \circ \cQ$. For that, pick $(p,q)\in \cQ'$ and let $r,s\in \Proj(\cM\bar\otimes \cB(\ell_2))$ be so that $(p,r )\not\in \cQ$ and $(s ,q)\not\in \cQ$. Let us show that $(r^\perp ,s^\perp )\in \psi^*(\cR)$.  Indeed, by our choice of $\cQ$, the assumptions on $r$ and $s$ imply that $p(\varphi \otimes 1)(\psi \otimes 1)(r)=0$ and $q(\varphi  \otimes 1)(\psi  \otimes 1)s=0$. Therefore,   as $(p,q)\in \cQ'$, we must have that  \[\big((\varphi\otimes 1)(\psi\otimes 1)(r^\perp) ,(\varphi\otimes 1)(\psi\otimes 1)(s^\perp)\big)\in \cQ'.\] As $\varphi^*(\cQ')=\cR$, this   implies that $((\psi\otimes 1)(r^\perp ),(\psi\otimes 1)(s^\perp ))\in \cR$; which in turn implies that  $(r^\perp , s^\perp )\in \psi^*(\cR)$. The arbitrariness of $r$ and $s$ shows that $(p,q)\in \cQ\circ \psi^*(\cR)\circ \cQ$. Therefore, as $\psi$ is quantum coarse, and as quantum coarse structures are closed under subrelations, this implies that $\cQ'\in \mathscr{Q}$. So  $\varphi$ is quantum expanding.
\end{proof}

We can now conclude that quantum coarse equivalence implies coboundedness and expansion, as promised.

\begin{theorem}\label{ThmEmbImpliesExpCob}
Every quantum coarse embedding is quantum expanding and every quantum coarse equivalence  is   quantum expanding and  quantum cobounded.
\end{theorem}

\begin{proof}
If $\varphi:\cM\to \cN$ is a quantum coarse equivalence, then the result follows straightforwardly from Proposition \ref{Prop19}. If $\varphi$ is a quantum coarse embedding, then, by definition, $\psi:\cM/\ker(\varphi)\to \cN$ is a quantum coarse equivalence and therefore, by Proposition \ref{Prop19}, it is quantum expanding. By Proposition \ref{PropSubspaceEquiv}, $\varphi$ is also quantum  expanding. 
\end{proof}

 \begin{remark}[On the backwards direction of Theorem \ref{ThmEmbImpliesExpCob}]
   In the beginning of this subsection, we mentioned that the reader should have no hope that a quantum coarse map $\cM\to \cN$ which is also quantum expanding and quantum cobounded should be a quantum coarse equivalence. Indeed, this can be easily seen since the inclusion $\C\to \cB(\ell_2)$ is a quantum function and, endowing $\C$ and $\cB(\ell_2)$ with their maximal intrinsic quantum coarse structures, i.e., $\mathscr{R}_\C=\mathrm{IQRel}(\C) $ and  $\mathscr{R}_{\cB(\ell_2)}=\mathrm{IQRel}(\cB(\ell_2)) $, it is clear that the inclusion $\C\to \cB(\ell_2)$ becomes quantum coarse, quantum expanding, and quantum cobounded. However, there is no quantum function $\cB(\ell_2)\to \C$.
   
   We can however say that every \emph{surjective} quantum coarse and  quantum expanding map is a quantum embedding,   and that every \emph{bijective} quantum coarse and quantum expanding map is a quantum equivalence. Indeed, say  $\varphi:\cM\to \cN$ is a surjective quantum coarse and quantum expanding map, then the induced map $\psi:\cM/\ker(\varphi)\to \cN$ is an isomorphism which, by Propositions \ref{PropTakeQuotCoarse} and  \ref{PropSubspaceEquiv}, is also quantum coarse and quantum expanding. Therefore, $\psi^{-1}$ must be quantum coarse, so   $\psi$ is a quantum coarse isomorphism, i.e., $\varphi$ is a quantum coarse embedding. \label{RemarkBack}
 \end{remark}

 \begin{theorem}\label{ThmEmbAndIsoHERE}
 Let $(\cM,\mathscr{V})$ and $(\cN,\mathscr{U})$ be quantum coarse spaces and let $\varphi:\cM\to \cN$ be a  surjective quantum function which is   quantum coarse and quantum expanding.
 \begin{enumerate}
     \item\label{Item1111111} If $\varphi$ is spatially implemented, then $\cstu(\cN,\mathscr{U})$ embeds into a hereditary subalgebra of  $\cstu(\cM,\mathscr{V})$.
     \item\label{Item11111112} There is a Hilbert space $K$ with $\mathrm{dens}(K)\leq \{\aleph_0,\mathrm{dens}(H_\cN)\}$ so that $\cstu(\cN,\mathscr{U})$ embeds into a hereditary subalgebra of  $\cstu(\cM\otimes 1_K,\mathscr{V} \bar\otimes \cB(K))$.
 \end{enumerate}
\end{theorem}
 
\begin{proof}
\eqref{Item1111111} Let $r\in \cM$ be a central projection with $\ker(\varphi)=(1-r)\cM$. By Remark \ref{RemarkBack}, the induced map $\psi:r\cM\to \cN$ is a spatially implemented quantum coarse isomorphism. Hence, by Theorem \ref{ThmEmbAndIso}(1),   $\cstu(\cN ,\mathscr{U})$ is spatially isomorphic to  $\cstu(r\cM ,r\mathscr{V}r) \simeq r\cstu(\cM ,\mathscr{V})r$.

\eqref{Item11111112} This item follows analogously, using Theorem \ref{ThmEmbAndIso}(2) and again reading the dimension estimate out of the proof of (\cite[Theorem IV.5.5 and Corollary IV.5.6]{Takesaki2002BookI}).
\end{proof}

\subsection{The metric case and quantum moduli}\label{SectionMetricModuli}
For metric spaces $(X,d)$ and $(Y,\partial)$, coarse maps are often defined in terms of the modulus of uniform continuity. Recall, if $f:X\to Y$, then its \emph{modulus of uniform continuity} is given by 
\[\omega_f(t)=\sup\{\partial(f(x),f(y)): \,  d(x,y)\leq t\}, \text{ for all } t\geq 0.\]
One can easily see that $f$ is coarse if and only if $\omega_f(t)<\infty$ for all $t\geq 0$.  Equivalently, if 
\[\tilde{\omega}_f(t)=\inf\{d(x,y): \,  d(f(x),f(y))\geq t\},\]
then $f$ is coarse if and only if $\lim_{t\to \infty}\tilde\omega(t)=\infty$ (see \cite[Lemma 3.1]{ChavezDominguezSwift2020}). We now see that, for quantum coarse metric spaces, our definition of quantum coarseness has an analogous characterization in terms of a modulus. 

We start recalling the definition of distance between projections introduced in 
\cite[Definition 2.6]{KuperbergWeaver2012MemAMS}:

\begin{definition} \label{Vdistance}
Let $(\cM,\mathbf{V}=(\cV_t)_{t\geq 0})$ be a quantum metric space and let $p,q\in\Proj(\cM\bar\otimes \cB(\ell_2))$. The \emph{$\mathbf{V}$-distance between $p$ and $q$} is defined by 
\[d_{\mathbf{V}}(p,q)=\inf\{t \in[0,\infty): \,  \exists a\in \cV_t, \ p(a\otimes 1)q\neq 0\}\]
(here we use the convention that the  infimum of an empty set is $\infty$).
\end{definition}

Our notion of quantum coarseness can be expressed as follows for the metric case (see Remark \ref{RemarkCDS}):

\begin{proposition}\label{PropModulusUniformCont}
Let $(\cM,\mathbf{V})$ and $(\cN,\mathbf{U})$ be quantum metric spaces and consider $\cM$ and $\cN$ as quantum coarse spaces endowed with $\mathscr{V}= \mathscr{V}_\mathbf{V}$ and $\mathscr{U}= \mathscr{V}_\mathbf{U}$, respectively. The following are equivalent for a quantum function $\varphi:\cM \to \cN$: 
\begin{enumerate}
    \item\label{ItemModulusUniformCont1} The map $\varphi$ is quantum coarse.
    \item\label{ItemModulusUniformCont2} There is $\omega:[0,\infty)\to [0,\infty)$ with $\lim_{t\to\infty}\omega(t)=\infty$ such that
    \[\omega(d_{\mathbf{V}}(p,q))\leq d_\mathbf{U}\big((\varphi\otimes 1)(p),(\varphi\otimes 1 )(q)\big)\]
     for all $p,q\in\Proj(\cM\bar\otimes \cB(\ell_2))$.
    \item\label{ItemModulusUniformCont3} We have  $\lim_{t\to \infty}\tilde{\omega}_\varphi(t)=\infty$, where  
    \[\tilde{\omega}_\varphi(t)=\inf\{d_\mathbf{U}((\varphi\otimes 1)(p),(\varphi\otimes 1)(q)): \,  d_\mathbf{V}(p ,q )\geq t\}\]
    and $ p $ and $q$ range over $\Proj(\cM\bar\otimes \cB(\ell_2))$.
\end{enumerate}
\end{proposition}

\begin{proof}
 Let $\mathbf{V}=(\cV_t)_{t\geq 0}$,  $\mathbf{U}=(\cU_{t})_{t\geq 0}$, and for each $t\geq 0$, let  $\cR_t=\cR_{\cV_t}$ and $\cQ_{t}=\cR_{\cU_t}$.

 The equivalence \eqref{ItemModulusUniformCont2} $\iff$ \eqref{ItemModulusUniformCont3} is completely straightforward. So we only show that  \eqref{ItemModulusUniformCont1} $\iff$ \eqref{ItemModulusUniformCont3}. For that, notice first that $\varphi $ is coarse if and only if for all $t>0$ there is $t'>0$ so that $\varphi^*(\cQ_t)\subset  \cR_{t'}$. Then, if  $t,t'\geq 0$, notice that, for  $p$ and $q$ ranging over $ \Proj(\cM\bar\otimes \cB(\ell_2))$, we have
 \begin{align*}
 \tilde{\omega}_\varphi(t')\geq t\ & \iff\  [d_{\mathbf{V}}(p ,q )\geq t'\ \Rightarrow \  d_{\mathbf{U}}( (\varphi\otimes 1)(p) ,(\varphi\otimes 1)(q) )\geq t]\\
 & \iff  \left[((\varphi\otimes 1)(p) ,(\varphi\otimes 1)(q) )\in\bigcup_{s<t} \cQ_s\ \Rightarrow \ (p ,q ) \in\bigcup_{s<t'} \cR_t \right]\\
 & \iff  \  \left[(p ,q ) \in\bigcup_{s<t} \varphi^*(\cQ_s)\Rightarrow \ (p ,q ) \in\bigcup_{s<t'}  \cR_s \right] \\
 &\iff \ \bigcup_{s<t}\varphi^*(\cQ_s)\subset  \bigcup_{s<t'} \cR_{s}.
 \end{align*}
As   $\tilde{\omega}_\varphi$ is increasing, we are done.
\end{proof}

\begin{remark}\label{RemarkCDS}
We point out that $\tilde\omega_\varphi$ was introduced in \cite[Definition 3.2]{ChavezDominguezSwift2020} with the small difference that there the projections $p$ and $q$ are only allowed to range over projections in $\cM$. However, those two definitions coincide for operator reflexive quantum metric spaces. Indeed, this can be seen for instance  from the proof of Proposition \ref{PropModulusUniformCont} and the fact that if  $\cV,\cV'\in \mathrm{QRel}(\cM\subset  \cB(H))$ are such that $(p\otimes 1,q\otimes 1)\in \cR_{\cV}$ implies $(p\otimes 1,q\otimes 1)\in \cR_{\cV'}$ for all $p,q\in \Proj(\cM)$, then $\cR_\cV\subset   \cR_{\overline{\cV'}} $.  
\end{remark}

We now turn to quantum expanding functions.   As introduced in  \cite[Definition 2.3]{ChavezDominguezSwift2020}, given $p\in \Proj(\cM$), we define the \emph{diameter of    $p$} as 
\[\diam_{\mathbf{V}}(p)=\sup\{d_\mathbf{V}(r,s): \,  \exists a\in \cB(H)\text{ so that }r(pap\otimes 1)s\neq 0\}.\]
Although we do not have an equivalent definition for quantum expanding maps in terms of a modulus, we can relate them with the following modulus introduced in \cite[Definition 3.2]{ChavezDominguezSwift2020}.

\begin{proposition}\label{PropModulusExp}
Let $(\cM,\mathbf{V})$ and $(\cN,\mathbf{U})$ be quantum metric spaces and consider $\cM$ and $\cN$ as quantum coarse spaces endowed with $\mathscr{V}= \mathscr{V}_\mathbf{V}$ and $\mathscr{U}= \mathscr{V}_\mathbf{U}$, respectively. If  $\varphi:\cM \to \cN$ is quantum expanding, then $\tilde\rho_{\varphi(t)}<\infty$ for all $t\geq 0$, where 
   \[\tilde{\rho}_\varphi(t)=\sup\{\diam_\mathbf{U}(\varphi(p)): \,  p\in \Proj(\cM)\text{ and } \diam_\mathbf{V}(p)\leq t\}.\]
\end{proposition}

\begin{proof}
Let $\mathbf{V}=(\cV_t)_{t\geq 0}$ and $\mathbf{U}=(\cU_{t})_{t\geq 0}$.  For each $t\geq 0$, let  $\cR_t=\cR_{\cV_t}$ and $\cQ_{t}=\cR_{\cU_t}$. Fix $t>0$. Let  \[\cQ^t=\bigcup_{\cQ\in (\varphi^*)^{-1}[\{\cR_t\}]}\cQ,\]   so, as a union of intrinsic quantum relations is an intrinsic quantum relation,  $\cQ^t\in \mathrm{IQRel}(\cN)$. Moreover,  $\varphi^*(\cQ^t)=\cR_t$, so, as $\varphi$ is quantum expanding, $\cQ^t\in\mathscr{Q}$.  Pick $t'>0$ so that $\cQ^t\subset  \cQ_{t'}$. 
 
 Let us show that $\tilde \rho_\varphi(t)\leq t'$.  
For each $p\in \Proj(\cM)$,  let
\[\cR(p)=\bigcup_{a\in \cB(H)}\{(r,s)\in \Proj(\cM\bar\otimes \cB(\ell_2)): \,  r(pap\otimes 1)s\neq 0\}\]
and
\[\cQ(p)=\bigcup_{a\in \cB(H)}\{(r,s)\in \Proj(\cN\bar\otimes \cB(\ell_2)): \,  r(\varphi(p)a\varphi(p)\otimes 1)s\neq 0\}.\]
So $\cR(p)$ and $\cQ(p)$ are intrinsic quantum relations on $\cM$ and $\cN$, respectively, and we have that 
 \begin{align*}
     \tilde{\rho}_{\varphi}(t)\leq t'\ &\iff\ \diam_{\mathbf{V}}(p)\leq t\ \Rightarrow\ \diam_{\mathbf{U}}(\varphi(p))\leq t'\\
     & \iff\ \cR(p)\subset  \cR_t\ \Rightarrow \ \cQ(p)\subset  \cQ_{t'},
     \end{align*}
     where the projections $p$ above range over $\Proj(\cM)$.
 
 Say $\cR(p)\subset  \cR_t$. By   \cite[Lemma 2.6]{ChavezDominguezSwift2020}, $\varphi^*(\cQ(p))\subset  \cR(p)$; so $\varphi^*(\cQ(p))\subset  \cR_t$. Therefore, $\varphi^*(\cQ(p)\cup\cQ^t)= \cR_t$ which, by the definition of $\cQ^t$, implies that $\cQ(p)\subset  \cQ^t$. By our choice of $t'$,  $\cQ(p)\subset  \cQ_{t'}$.
\end{proof}

We do not know if the condition in Proposition \ref{PropModulusExp} characterizes quantum expansion. 

\begin{remark}
We notice that quantum coarse embeddings between quantum \emph{metric} spaces were introduced differently in \cite{ChavezDominguezSwift2020}:  according to \cite[Definition 3.4]{ChavezDominguezSwift2020}, a quantum  function $\varphi:\cM\to \cN$ is a \emph{quantum coarse embedding} if $\lim_{t\to \infty}\tilde\omega_\varphi(t)=\infty$ and $\tilde\rho_\varphi(t)<\infty$. 
\end{remark}

 \subsection{Quantum asymptotic dimension}\label{SubsectionAsyDim}
Asymptotic dimension was introduced by M. Gromov in \cite[Section 1.E]{Gromov1993LondonLecNotes},  and it has since become one of the main concepts in coarse geometry. In this section, we  quantize this notion and show that, just as in the classical setting, quantum asymptotic dimension is preserved under quantum coarse embeddings. In the case of quantum metric spaces, the development below coincides with \cite[Theorem 4.6]{ChavezDominguezSwift2020}.

\begin{definition}
Let $(\cM,\mathscr{R})$ be an intrinsic quantum coarse space. 
\begin{enumerate}
\item Given $\cP\subset  \Proj(\cM)$, we say that $\cP$ \emph{covers $\cM$} if $1_{H_\cM}=\bigvee_{p\in \cP}p$.
    \item Given $\cR\in \mathscr{R}$ and $\cP\subset \Proj(\cM)$, we say that $\cP$  is \emph{$\cR$-disjoint} if $(p\otimes 1,q\otimes 1)\not\in \cR$ for all $p,q\in \cP$ with $p\neq q$.
    \item Given $\cR\in \mathscr{R}$ and $p\in \Proj(\cM)$, we say that the \emph{diameter of $p$ is at most $\cR$}, and we write $\diam(p)\leq \cR$, if $(r,s)\in \cR$ for all $r,s\in \Proj(\cM)$ such that there is $a\in \cB(H_\cM)$ with $r(pap\otimes 1)s\neq 0$.\footnote{This is a clear adaptation of the quantity  $\diam_\mathbf{V}(p)$ introduced in  \cite[Definition 2.3]{ChavezDominguezSwift2020} (and used in Section \ref{SectionMetricModuli}) for the metrizable case.}
    \item We say $\cP\subset  \Proj(\cM)$ is \emph{uniformly bounded} if there is $\cR\in \mathscr{R}$ so that $\diam(p)\leq \cR$ for all $p\in \cP$.
\end{enumerate}
If $(\cM,\mathscr{V})$ is a quantum coarse space, all the definitions above are made with respect to $\mathscr{R}_\mathscr{V}$.
\end{definition}

\begin{definition}
Let $(\cM,\mathscr{V})$ be a  quantum coarse space and $n\in\N\cup\{0\}$. We say that $(\cM,\mathscr{V})$ has \emph{asymptotic dimension at most $n$}  if for all $\cR\in \mathscr{R}_\mathscr{V}$ there are $\cP_0,\cP_1,\ldots, \cP_n\subset  \Proj(\cM)$ so that 
\begin{enumerate}
    \item $(\cP_i)_{i=0}^n$ covers $\cM$,
    \item each $\cP_i$ is $\cR$-separated, and 
    \item each $\cP_i$ is uniformly bounded.
\end{enumerate}
We say that $(\cM,\mathscr{V})$ has \emph{asymptotic dimension at most $n$}, and  write $\mathrm{asydim}(\cM,\mathscr{V})=n$, if $n$ is the smallest element in $\N\cup\{0\}$ satisfying the above. If no such $n$ exists, we say that $(\cM,\mathscr{V})$ has \emph{infinite asymptotic dimension}.   
\end{definition}

It is clear that the definition above coincides with the usual definition of asymptotic dimension of a coarse space $(X,\cE)$ for $\cM=\ell_\infty(X)$ and $\mathscr{V}=\mathscr{V}_{\cE}$.

\begin{theorem}\label{ThmAsyDim}
Let $(\cM,\mathscr{V})$ and $(\cN,\mathscr{U})$ be a  quantum coarse spaces. If there is a quantum coarse and quantum expanding map $\varphi:\cM\to \cN$,    then $\mathrm{asydim}(\cN,\mathscr{U})\leq \mathrm{asydim}(\cM,\mathscr{V})$. In particular, if and $(\cM,\mathscr{V})$ and $(\cN,\mathscr{U})$ are quantum coarsely equivalent, then $\mathrm{asydim}(\cN,\mathscr{U})= \mathrm{asydim}(\cM,\mathscr{V})$. 
\end{theorem}

\begin{proof}
Let $\varphi:\cM\to \cN$ be a quantum coarse and quantum expanding  map. First notice that if  $\cQ\in \mathscr{R}_\mathscr{U}$ and $\cP\subset  \Proj(\cM)$ is $\varphi^*(\cQ)$-disjoint, then  $\varphi[\cP]$ is $\cQ$-disjoint.   Now let $\cR\in \mathscr{R}$ and let $\cQ$ be the union of all   $\cQ'\in \mathrm{IQRel}(\cN)$ so that $\varphi^* (\cQ')\subset  \cR$. As $\varphi$ is quantum expanding, we have that $\cQ\in \mathscr{R}_{\mathscr{U}}$. Proceeding analogously as in the proof of Proposition \ref{PropModulusExp}, we have that if $\diam (p)\leq \cR$, then $\diam(\varphi(p))\leq \cQ$.

By the above, if $\cQ\in \mathscr{R}_{\mathscr{U}}$ and   $\cP_0,\cP_1,\ldots, \cP_n\subset  \Proj(\cM)$ forms a $\varphi^*(\cQ)$-disjoint, uniformly bounded cover of $\cM$, then $(\varphi[\cP_i])_{i=0}^n$ is a $\cQ$-disjoint, uniformly bounded cover of $\cN$. 
\end{proof}

\begin{acknowledgements}
The authors would like to thank Nik Weaver for useful comments on a previous version of this paper. The first-named author would also like to thank Ben Hayes for enlightening  conversations about von Neumann algebras.
\end{acknowledgements}
   
\bibliographystyle{alpha}
\bibliography{bibliography}

\end{document}